\documentclass[onefignum,onetabnum,final]{siamart190516}

\usepackage{lipsum}
\usepackage{amsfonts}
\usepackage{graphicx}
\usepackage{epstopdf}
\ifpdf
  \DeclareGraphicsExtensions{.eps,.pdf,.png,.jpg}
\else
  \DeclareGraphicsExtensions{.eps}
\fi

\newsiamremark{hypothesis}{Hypothesis}
\crefname{hypothesis}{Hypothesis}{Hypotheses}
\newsiamremark{rem}{Remark}
\newsiamthm{claim}{Claim}
\newsiamthm{observation}{Observation}
\newsiamthm{ass}{Assumption}

\headers{Efficient Search of First-Order Nash Equilibria}{D. M. Ostrovskii, A. Lowy, and M. Razaviyayn}

\title{Efficient Search of First-Order Nash Equilibria \\ 
	in Nonconvex-Concave Smooth Min-Max Problems 
	}

\author{Dmitrii M.~Ostrovskii\thanks{Viterbi School of Engineering, University of Southern California, Los Angeles, CA 90089, USA.
	Email: \email{dostrovs@usc.edu}, \email{lowya@usc.edu}, \email{razaviya@usc.edu}.}
\and  Andrew Lowy\footnotemark[1]
\and Meisam Razaviyayn\footnotemark[1]}

\usepackage{microtype}
\usepackage{xfrac}
\usepackage{enumitem}
\usepackage{comment}

\usepackage{color}
\usepackage[numbers]{natbib}

\usepackage{amsmath,amssymb}
\usepackage{nccmath}
\usepackage{mathrsfs}
\usepackage{enumitem}

\usepackage{wrapfig}

\usepackage{url}

\usepackage{algorithm}
\usepackage{algpseudocode}

\usepackage{pdfpages}
\usepackage{multirow}
\usepackage{dsfont,makecell}
\usepackage{color, colortbl}

\usepackage{amsopn}
\DeclareMathOperator*{\Argmin}{Argmin}
\DeclareMathOperator*{\Argmax}{Argmax}
\DeclareMathOperator*{\argmin}{argmin}
\DeclareMathOperator*{\argmax}{argmax}

\newcommand{\Dx}{\textsf{\textup S}_{X}}
\newcommand{\Dy}{\textsf{\textup S}_{Y}}
\newcommand{\Dz}{\textsf{\textup S}_{Z}}
\newcommand{\Dxw}{\textsf{\textup S}_{X,\omega}}
\newcommand{\Dzw}{\textsf{\textup S}_{Z,\omega}}
\newcommand{\Exw}{\textsf{\textup W}_{X,\omega}}
\newcommand{\Ezw}{\textsf{\textup W}_{Z,\omega}}
\newcommand{\Ex}{\textsf{\textup W}_{X}}
\newcommand{\Ey}{\textsf{\textup W}_{Y}}
\newcommand{\Ez}{\textsf{\textup W}_{Z}}
\newcommand{\hEy}{\wh{\textsf{\textup W}}_{\y}}
\newcommand{\hDy}{\wh{\textsf{\textup S}}_{\y}}
\newcommand{\envalias}[2]{%
  \expandafter\let\csname#1\expandafter\endcsname\csname#2\endcsname
  \expandafter\let\csname end#1\expandafter\endcsname\csname end#2\endcsname
}
\envalias{eq}{equation}
\envalias{eq*}{equation*}
\envalias{ald}{aligned}
\newcommand{\x}{{\textsf{\textup x}}}
\newcommand{\y}{{\textsf{\textup y}}}
\newcommand{\XX}{\mathcal{X}}
\newcommand{\YY}{\mathcal{Y}}
\newcommand{\ZZ}{\mathcal{Z}}
\newcommand{\gx}{\nabla_\x}
\newcommand{\gy}{\nabla_\y}
\newcommand{\py}{\partial_\y}
\newcommand{\Lxx}{L_{\x\x}}
\newcommand{\Lyy}{L_{\y\y}}
\newcommand{\Lxy}{L_{\x\y}}

\newcommand{\bx}{\wt{x}}
\newcommand{\by}{\bar{y}}
\newcommand{\ex}{\veps_\x}
\newcommand{\ey}{\veps_\y}
\newcommand{\bey}{\ey}

\newcommand{\prox}{\textup{prox}}
\newcommand{\gam}{\gamma}
\newcommand{\lam}{\lambda}

\newcommand{\Rx}{R_{\x}}
\newcommand{\Ry}{R_{\y}}

\newcommand{\FGM}{$\textsf{FGM}$}
\newcommand{\FGMR}{$\textsf{RestartFGM}$}
\newcommand{\AppGrad}{$\textsf{SolveRegDual}$}
\newcommand{\proofstep}[1]{$\boldsymbol{{#1}^o}$}
\newcommand{\bF}{F^{\reg}}
\newcommand{\reg}{\textup{reg}}
\newcommand{\vphi}{\varphi}

\newcommand{\half}{\frac{1}{2}}

\newcommand{\FullData}{(\ex,\ey,\Lxx,\Lxy,\Lyy,\Ry,\Gap)}

\newcommand{\eps}{\epsilon}
\newcommand{\Gap}{\Delta}
\newcommand{\Tx}{T_\x}
\newcommand{\Ty}{T_\y}
\newcommand{\bTy}{\overline{T}_\y}
\newcommand{\bTx}{\overline{T}_\x}
\newcommand{\Sy}{S_\y}
\newcommand{\lamy}{\lam_{\y}}
\newcommand{\gamx}{\gam_\x}
\newcommand{\gamy}{\gam_\y}
\newcommand{\To}{T^{o}}
\newcommand{\So}{S^{o}}

\renewcommand{\le}{\leqslant}
\renewcommand{\ge}{\geqslant}
\newcommand{\proj}{\Pi}
\newcommand{\bvphi}{\vphi^{\reg}}

\algrenewcommand\textproc{\textsf}

\newcommand{\nn}{\notag\\}

\newcommand{\ba}{\begin{array}{c}}
\newcommand{\bal}{\begin{array}{l}}
\newcommand{\ea}{\end{array}}

\newcommand{\R}{\mathds{R}}

\newcommand{\bit}{\begin{itemize}}
\newcommand{\eit}{\end{itemize}}

\newcommand{\bvec}{\left(\!\!\!\begin{array}{c} }
\newcommand{\evec}{\end{array}\!\!\!\right)}

\newcommand{\wt}{\widetilde}

\newcommand{\barr}{\begin{array}}
\newcommand{\earr}{\end{array}}

\newcommand{\odima}[1]{#1}

\newcommand{\wh}{\widehat}
\newcommand{\veps}{\varepsilon}

\newcommand{\lang}{\left\langle}

\newcommand{\rang}{\right\rangle}
\newcommand{\dgf}{d.-g.~f.}

\renewcommand{\it}{{\boldsymbol{i}}}

\usepackage{cleveref}
\crefname{algorithm}{Algorithm}{Algorithms}
\crefname{assumption}{Assumption}{Assumptions}
\crefname{equation}{}{}
\crefname{figure}{Fig.}{Figs.}
\crefname{table}{Table}{Tables}
\crefname{section}{Section}{Sections}
\crefname{subsection}{Section}{Sections}
\crefname{theorem}{Theorem}{Theorems}
\crefname{lemma}{Lemma}{Lemmmas}
\crefname{proposition}{Proposition}{Propositions}
\crefname{definition}{Definition}{Definitions}
\crefname{corollary}{Corollary}{Corollaries}
\crefname{remark}{Remark}{Remarks}
\crefname{example}{Example}{Examples}
\crefname{appendix}{Appendix}{Appendices}

\begin{document}
\maketitle

\begin{abstract}
We propose an efficient algorithm for finding first-order Nash equilibria in min-max problems of the form $\textstyle\min_{x \in X}\max_{y\in Y} F(x,y)$, where the objective function is smooth in both variables and concave with respect to~$y$; the sets~$X$ and~$Y$ are convex and ``projection-friendly,'' and~$Y$ is compact.
Our goal is to find an $(\ex,\ey)$-first-order Nash equilibrium with respect to a stationarity criterion that is stronger than the commonly used proximal gradient norm. 
The proposed approach is fairly simple: we perform approximate proximal-point iterations on the primal function, with inexact oracle provided by Nesterov's algorithm run on the regularized function~$F(x_t,\cdot)$, 
$x_t$ being the current primal iterate.
The resulting iteration complexity is~$O({\ex}^{-2} \, {\ey}^{-1/2})$ up to a logarithmic factor. 
As a byproduct, the choice~$\ey = O(\ex^2)$ allows for the~$O({\ex}^{-3})$ complexity of finding an~$\ex$-stationary point for the standard Moreau envelope of the primal function. 
Moreover, when the objective is strongly concave with respect to~$y$, the complexity estimate for our algorithm improves to~$O({\ex}^{-2}{\kappa_\y}^{1/2})$ up to a logarithmic factor, where~$\kappa_\y$ is the condition number \odima{appropriately adjusted for coupling}.
In both scenarios, the complexity estimates are the best known so far, and are only known for the (weaker) proximal gradient norm criterion.
\odima{Meanwhile, our approach is ``user-friendly'':
(i) the algorithm is built upon running a variant of Nesterov's accelerated algorithm as subroutine and avoids extragradient steps; 
(ii) the convergence analysis recycles the well-known results on accelerated methods with inexact oracle.}
 Finally, we extend the approach to non-Euclidean proximal geometries. 
\end{abstract}

\section{Introduction}
\label{sec:problem:setup}
In recent years, min-max problems have received significant attention across the optimization and machine learning communities due to their applications in  training generative adversarial networks (GANs) \cite{goodfellow2014GAN}, training machine learning models that are robust to adversarial attacks \cite{madry2019adversarial}, reinforcement learning \cite{dai2018RL}, fair statistical inference \cite{baharlouei2020fair}, and distributed non-convex optimization \cite{lu2019block}, to name a few. These applications involve solving optimization problems in the general form
 \begin{eq}
\label{opt:min-max}
\min_{x \in X} \max_{y \in Y} F(x,y), 
\end{eq}
where~$F$ is a smooth objective function and~$X,Y$ is a pair of convex sets \odima{in the corresponding Euclidean spaces~$\XX, \YY$.}
When the objective is convex in $x$ and concave in $y$, ~\cref{opt:min-max} is well-studied. In this case, the corresponding variational inequality is monotone, and there is a number of efficient algorithms known for solving it, even at the optimal rate (see, e.g., \cite{nemirovski2004prox}, \cite{nesterov2005smooth}, \cite{ouyang}). However, many of the applications discussed above involve an objective $F$ that is nonconvex in $x$ and not necessarily concave in~$y$, which makes the problem much harder to solve.
In fact, even Nash equilibria are not guaranteed to exist in~\cref{opt:min-max} in this general nonconvex-nonconcave setting.

In this work, we study~\eqref{opt:min-max} under the assumption that $F(x,y)$ is concave in $y$ but
do not assume convexity $x$.  To the best of our knowledge,~\cite{nouiehed2019solving} was the first work providing non-asymptotic convergence rates for  nonconvex-concave problems without assuming special structure of the objective function. 
They use the notion of $\varepsilon$-first order Nash equilibrium (FNE) to measure the rate of convergence of their algorithm. This notion looks at the min-max problem as a two-player zero-sum game and uses the first-order optimality condition with respect to each variable as the optimality measure. Using this notion, they showed that their algorithm finds an~$\varepsilon$-first-order Nash equilibrium in~$O({\varepsilon}^{-3.5})$ gradient evaluations. 
In this work, we use a similar optimality notion to the one in \cite{nouiehed2019solving}; however, in order to  measure the optimality w.r.t. the $x,y$ variables separately, we generalize their notion to $(\ex,\ey)$-first order stationarity. Therefore, by setting $\ex= \ey = \varepsilon$, we obtain the optimality measure used in \cite{nouiehed2019solving}. Using the defined first order stationarity measure, we propose an algorithm that can find $(\ex,\ey)$-first order stationary point in $O({\ex}^{-2} {\ey}^{-1/2})$ gradient evaluations.

\odima{
Another way to measure the convergence rate of an algorithm for solving \eqref{opt:min-max} is to define the primal function~$\vphi(x) = \max_{y\in Y} F(x,y)$, and measure the first-order optimality in terms of the nonconvex problem $\min_{x \in X} \vphi(x)$. 
In this context, a commonly used inaccuracy measure for a candidate solution~$\wh x$ is the gradient norm of the standard Moreau envelope of the primal function (see Section~\ref{sec:moreau} for more details).
Using this viewpoint, subtle analyses have been provided in~\cite{thekumparampil2019efficient,kong2019accelerated}, and more recently, in the concurrent work~\cite{lin2020near} whose preprint was announced a few days prior to ours.
More recently, a similar approach has been used in~\cite{zhao2020primal}.
The underlying idea in all these works, as well as in ours, is to obtain the next iterate~$(x_{t+1}, y_{t+1})$ by approximately solving a strongly-convex-concave saddle-point problem
\begin{eq}
\label{eq:minmax-conceptual-updates-intro}
\min_{x \in X} \max_{y \in Y} [F(x,y) + \Lxx \|x-x_{t}\|^2],
\end{eq}
where~$\Lxx$ is the uniform over~$y \in Y$ bound on the Lipschitz constant of~$\gx F(\cdot,y)$. 
Yet, there are notable differences between all these works, which we shall now discuss.
}

\odima{
\subsection{Related work}
The work~\cite{thekumparampil2019efficient} focuses on the problem of finding an~$\ex$-stationary point of the Moreau envelope~$\vphi_{2\Lxx}(x) := \min_{x' \in X} [\vphi(x') + \Lxx \|x'-x\|^2]$.
To achieve this goal, they solve~\cref{eq:minmax-conceptual-updates-intro} up to accuracy~$\eps$ {in objective value}. 
The resulting scheme produces a point~$\wh x$ satisfying~$\|\vphi_{2\Lxx}(\wh x)\| \le \ex$ in~$O(\ex^3)$ oracle calls, provided that one takes~$\eps = O(\ex^2)$. 
However, they only handle the case~$X = \XX$ and do not provide an algorithm to reach an~$(\ex, \ey)$-FNE for general~$(\ex,\ey)$; modifying their scheme correspondingly might be challenging because of~$Y \ne \YY$. 
(Our discussions in~\cref{sec:moreau} shed some light on such intricacies).
Moreover, their proposed algorithm is way less transparent than ours: it comprises an extragradient-type scheme {\em as well as} Nesterov's acceleration, whereas our approach is based solely on the analysis of Fast Gradient Method (FGM) a version of Nesterov's accelerated algorithm, and uses the readily available results for inexact-oracle FGM due to~\cite{devolder2014first}. 

The approach in the concurrent work~\cite{lin2020near} more closely resembles ours. 
In particular, their scheme also avoids using an extragradient-type subroutine, and produces an~$(\ex,\ey)$-FNE in~$O(\ex{}^{-2}\ey{}^{-1/2})$ oracle calls; setting~$\ey = O(\ex^2)$ then allows to recover the~$O(\ex^{-3})$ result for the Moreau envelope. 
However, they use the definition of~$(\ex,\ey)$-FNE based on the proximal gradient norm, or the {\em weak criterion} in our terminology in~\cref{sec:setup} (cf.~\cref{eq:proximal-gradient}), whereas our scheme produces an~$(\ex,\ey)$-FNE with respect to the so-called {\em strong criterion} (cf.~\cref{def:our-measure}). As we discuss in~\cref{rem:criteria-comparison}, the latter task is more challenging: an~$(\ex,\ey)$-FNE with respect to the strong criterion is also such with respect to the weak criterion (hence the names); meanwhile, a guarantee on the weak criterion does not imply {\em any guarantee} on the strong one in the presense of constraints. 
Moreover, the~$\ey = O(\ex^2)$ reduction for the Moreau envelope in~\cite{lin2020near} does not allow for~$X \ne \XX$ (same as in~\cite{thekumparampil2019efficient}). 
In addition, this reduction relies on the result~\cite[Prop.~4.12]{lin2019gradient}; as we discuss in~\cref{sec:moreau}, this result seems to be invalid unless~$Y = \YY$, which is irrelevant in the context of~\cref{opt:min-max}. 
We rectify both these issues in~\cref{sec:moreau} through the delicate use of the strong criterion.
It should be noted, however, that the authors of~\cite{lin2020near} focus on the convex-concave scenario which we do not address here.
 
The work~\cite{kong2019accelerated} also establishes an~$O(\ex{}^{-2}\ey{}^{-1/2})$ complexity to find an~$(\ex,\ey)$-first-order stationary point; however, to the best of our understanding, their stationarity criterion is not directly comparable to ours nor to the one in~\cite{lin2020near}.
Setting~$\ey = O(\ex^2)$, the authors of~\cite{kong2019accelerated} obtain~$O(\ex^{-3})$ complexity result for the criterion similar to the gradient norm of the Moreau envelope, but slightly weaker (as follows by comparing~\cite[Eq.~(2)]{kong2019accelerated} with the second claim of our~\cref{prop:moreau-to-primal}).
However, they assume direct access to the gradient of the smooth function~$\max_{y \in Y}[F(x,y) - \lam\|y\|^2]$, which is unrealistic (e.g., we solve a similar task via an FGM subroutine). 
This assumption has been removed in the recent work~\cite{zhao2020primal} (that appeared some time after our work). 
Moreover, while~\cite{zhao2020primal} only focuses on the primal accuracy measure (in the same sense as~\cite{kong2019accelerated}), 
they address the general Bregman geometries (which we also do, see~\cref{sec:non-euclidean}). 
However,~\cite{zhao2020primal} do not address the general~$(\ex,\ey)$-stationarity notion.

Finally, let us discuss the precursor work~\cite{nouiehed2019solving}. 
In our terminlogy,~\cite{nouiehed2019solving} shows an~$O(\ex{}^{-2}\ey{}^{-3/2})$ complexity estimate to find an $(\ex, \ey)$-FNE, using a notion of stationarity weaker than the one in~\cite{lin2020near} (and thus {\em a fortiori} weaker than the one we use in the present work).
The extra~$\ey^{-1}$ complexity factor compared to our result comes from using a more naive algorithmic approach: instead of forming an iterate sequence by solving~\cref{eq:minmax-conceptual-updates-intro}, they proceed by running projected gradient descent on the smoothed primal function~$\max_{y \in Y} [F(x,y) - \lamy \|y\|^2].$ 
Taking~$\lam = \ey/\Ry$ ensures that~$\gy [F(x,y) - \lamy \| y \|^2] \approx \gy F(x,y)$ up to~$O(\ey)$ error, which guarantees that~$(\ex,\ey)$-FNE for the new problem is still valid for the initial one. 
However, gradient descent suffers from a poor smoothness of the smoothed primal function, whose gradient is only Lipschitz with modulus~$\Lxx + O(\Lxy^2/\lamy) = O(\ey^{-1})$,
which results in the final iteration complexity estimate~$O(\ex{}^{-2} \ey{}^{-3/2})$.
}

\section{Problem formulation and preview of the main result}
\label{sec:setup}
We study the min-max problem \cref{opt:min-max} in the setting where~$X,Y$ are convex and ``projection-friendly'' sets with non-empty interior in the corresponding Euclidean spaces~$\XX, \YY$; moreover,~$Y$ is contained in a Euclidean ball with radius~$\Ry < \infty$. 
The function~$F: X \times Y \to \R$ is concave  in~$y$ for all~$x \in X$, and has Lipschitz gradient, namely, 
the inequalities 
\begin{eq}
\label{ass:smoothness}
\begin{ald}
\|\gx F(x',y) - \gx F(x,y) \| &\le \Lxx \|x' - x\|, \quad
 \|\gy F(x,y') - \gy F(x,y) \| \le \Lyy \|y' - y\|, \\
\|\gx F(x,y') - \gx F(x,y) \| &\le \Lxy \|y' - y\|, \quad 
\|\gy F(x',y) - \gy F(x,y) \|  \le \Lxy \|x' - x\|
\end{ald}
\end{eq}
hold uniformly over~$x, x' \in X$ and~$y,y' \in Y$ with Lipschitz constants~$\Lxx,\Lyy,\Lxy$. 
Here and in what follows,~$\|\cdot\|$ and~$\lang \cdot, \cdot \rang$ denote the standard Euclidean norm and inner product (regardless of the space), and~$[\gx F(x,y),\gy F(x,y)]$ are the components of the full gradient~$\nabla F(x,y)$. 
Instead of seeking an exact solution to \eqref{opt:min-max}, we focus on the more feasible task of finding an \textit{approximate first-order Nash equilibrium}.
\begin{definition}
\label{def:stationarity}
A point~$(\wh x, \wh y) \in X \times Y$ is called {\em $(\ex,\ey)$-approximate first-order Nash equilibrium ($(\ex,\ey)$-FNE)} in the problem~\eqref{opt:min-max}
if the following holds:
\begin{eq}
\label{eq:stationarity}
\begin{ald}
\Dx(\wh x, \gx F(\wh x, \wh y), \Lxx) \le \ex \quad \text{and} \quad
\Dy(\wh y, -\gy F(\wh x, \wh y), \Lyy) \le \ey,
\end{ald}
\end{eq}
where the inaccuracy measure~$\Dz$, with~$Z$ being a convex subset of a Euclidean space~$\ZZ$, is defined on triples~$z,\zeta \in Z$,
$L \ge 0$ as follows:
\begin{align}
\label{def:our-measure}
\Dz^2(z, \zeta, L) := 2L \max_{z' \in Z} \left[-\lang \zeta, z' - z \rang  - \tfrac{L}{2}\|z'-z\|^2 \right].
\end{align}
\end{definition}
\begin{rem}
\label{rem:criteria-comparison}
A more common stationarity measure in the context of constrained minimization of a convex function~$f: Z \to \R$ is the norm of the {\em proximal gradient}, that is~$\Ez(\wh z, \nabla f(\wh z), L) := L \left\| \wh z - \proj_{Z}\left[\wh z - \tfrac{1}{L} \nabla f(\wh z) \right] \right\|,$ where~$\Pi_Z(\cdot)$ is the operator of Euclidean projection onto~$Z$ (see~\cite{nesterov2013introductory}), and we define the functional
\begin{eq}
\label{eq:proximal-gradient}
\Ez(z, \zeta, L) := L \left\| z - \proj_{Z}\left[z - \tfrac{1}{L} \zeta \right] \right\|
\end{eq}
for convenience.
In the unconstrained case, both measures reduce to~$\|\nabla f(\wh z) \|$. 
However, in the general constrained case~$\Dz$ provides a stronger criterion, in the following sense:
\begin{itemize}
\item[(i)]
For any~$z,\zeta,L$ one has $\Ez(z,\zeta,L) \le \Dz(z,\zeta,L)$; thus, any~$(\ex,\ey)$-FNE in the sense of~\cref{def:stationarity} is an~$(\ex, \ey)$-FNE  in the weak sense, i.e., $\Ex(\wh x, \gx F(\wh x, \wh y), \Lxx) \le \ex, \Ey(\wh y, -\gy F(\wh x, \wh y), \Lyy) \le \ey$. See~\cite[Thm. 4.3]{barazandeh2020solving}.
\item[(ii)] 
\odima{The converse is generally false (unless in the unconstrained case). In particular, in~\cref{sec:moreau} (cf.~\cref{rem:moreau-tightness}) we exhibit a minimization problem in which~$\Ez(\wh z, \nabla f(\wh z), L) \le \veps$ at~$\wh z \in Z$, but~$\Dz(\wh z, \nabla f(\wh z), L)$ is arbitrarily large.}
\end{itemize}
\end{rem}
 



Our goal is to provide an efficient algorithm for finding~$(\ex,\ey)$-FNE\footnotemark
\footnotetext{Exact~FNE might not exist when~$X$ is not compact; however~$(\ex,\ey)$-FNE exists for all~$\ex, \ey > 0$.}
given access to the full gradient oracle~$\nabla F(x,y)$. 
Following the established trend in the literature, we assume the feasible sets~$X,Y$ to be ``projection-friendly'', i.e.,~ Euclidean projection onto them can be done with a small computational effort; thus, the natural notion of efficiency is simply the number of gradient computations.
Besides the two accuracies~$\ex,\ey$, the Lipschitz parameters~$\Lxx, \Lyy, \Lxy$, and the ``radius''~$\Ry$ of~$Y$, 
we need a parameter quantifying the hardness of the primal problem -- that of minimizing 
\begin{eq}
\label{def:prim-function}
\vphi(x) := \max_{y \in Y} F(x,y).
\end{eq} 
Since~$X$ can be unbounded, the natural choice of such parameter is the \textit{primal gap}~$\Gap$,
\[
\Gap := \vphi(x_0) - \min_{x \in X} \vphi(x),
\] 
where~$x_0$ is the initial iterate. 
To give a concise and intuitive statement of our main result, it is helpful to define the ``coupling-adjusted'' counterpart~$\Lyy^+$ of~$\Lyy$, defined as
\begin{eq}
\label{def:lyy-plus}
\Lyy^+ := \Lyy + \frac{\Lxy^2}{\Lxx},
\end{eq}
as well as the unit-free quantities -- the ``complexity factors''~$\Tx$ and~$\Ty$ given by
\begin{eq}
\label{eq:complexity-factors}
\Tx := \frac{\Lxx \Gap}{\ex^2}, \quad \Ty := \sqrt{\frac{\Lyy^+ \Ry}{\ey}}.
\end{eq}
Upon consulting the literature (e.g.,~\cite{carmon2017lower,nesterov2012howto}), we recognize~$\Tx$ as the iteration complexity of finding~$\ex$-stationary point (with respect to gradient norm) in the class of unconstrained minimization problems with~$\Lxx$-smooth (possibly nonconvex) objective and initial gap~$\Gap$.
On the other hand, we recognize~$\Ty$ as the tight complexity bound for the problem of finding~$\ey$-stationary point in the class of maximization problems with {concave} and~$\Lyy^+$-smooth objective, given the initial point within~$\Ry$ distance of an optimum, using first-order information. (This bound is also tight in the constrained setup, with~$\ey$ bounding the proximal gradient norm.)
We now state our main result.

\begin{theorem}[Abridged formulation of~\cref{th:upper-bound}]
\label{th:upper-bound-simple}
There exists an algorithm that, given~$\FullData$, outputs~$(\ex,\ey)$-FNE of the problem~\eqref{opt:min-max} in
\begin{eq}
\label{eq:complexity-simple}
\wt O\left(\Tx \Ty \right) 
\end{eq}
computations of~$\nabla F(x,y)$ and projections, where~$\wt O(\cdot)$ hides logarithmic factors in~$\Tx^{\vphantom+},\Ty$. 
\end{theorem}

This result merits some comments.
First, from~\cref{eq:complexity-factors}-\cref{eq:complexity-simple} we see that the complexity of finding~$(\ex,\ey)$-FNE in the problem~\eqref{opt:min-max} can be viewed as the product of the ``primal'' complexity of finding an~$\ex$-stationary point of~$F(\cdot, y)$ with fixed~$y$, and the ``dual'' complexity of finding~$\ey$-stationary point of~$\psi_x(\cdot) = \min_{x' \in X} F(x',\cdot)+\Lxx \|x'-x\|^2$ with fixed~$x \in X$ on~$\Ry$. 
\odima{Note that~$\psi_x(\cdot)$ has~$\Lyy^+$-Lipschitz gradient by Danskin's theorem~(see~\cite{danskin1966theory} and~\cite[Lem.~24]{nouiehed2019solving}), and is associated with the standard Moreau envelope~$\vphi_{2\Lxx}(x) := \textstyle\min_{x' \in X} [\vphi(x') + \Lxx \|x'-x\|^2]$ of the primal function~$\vphi(x)$ (\cite{jin2019local}).}

\odima{
Second, in~\cref{sec:moreau} we prove that the primal component~$\wh x$ of an~$(\ex,\ey)$-FNE with~$\ey =\ex^2/(\Lxx\Ry)$ satisfies~$\|\vphi_{2\Lxx}(\wh x)\| = O(\ex)$. 
In view of~\cref{th:upper-bound-simple}, this leads to the complexity~$\wt O(\ex^{-3})$ of finding an~$\ex$-stationary point for the standard Moreau envelope (see~\cref{sec:moreau} for the rigorous result (cf.~\cref{eq:moreau-complexity-final}) and detailed discussion). 
Such a result is known from the recent literature~\cite{thekumparampil2019efficient,lin2020near} in the case~$X = \XX$; 
moreover, as we discuss in~\cref{sec:moreau}, the result~\cite[Prop.~4.12]{lin2019gradient} that is commonly used (in particular, in~\cite{lin2020near}) in order to reduce the Moreau envelope criterion to the criterion based on the norm of the proximal gradient (i.e., our weak criterion in~\cref{eq:proximal-gradient}) seems to be invalid when~$Y \ne \YY$.
Our results close these gaps by working with the strong criterion~\cref{def:our-measure}.
}

Third, our approach 
can be  extended to composite objectives,~e.g., by following~\cite{nesterov2013first,optbook1}. 
To keep the presentation simple, we avoid such extension here (see, e.g., \cite{barazandeh2020solving}). 
On the other hand, extension to non-Euclidean geometries faces some non-trivial challenges that have not been properly addressed in the prior literature.\footnotemark
~In~\cref{sec:non-euclidean} we discuss these challenges and introduce the necessary adjustments into our framework.
\footnotetext{A notable exception is the work~\cite{zhao2020primal} that appeared shortly after the first version of this manuscript.}

\paragraph{Notation} 
Throughout the paper, and unless explicitly stated otherwise,~$\|\cdot\|$ and~$\lang \cdot, \cdot \rang$ denote the standard Euclidean norm and inner product regardless of the (Euclidean) space.~We let~$[T] := \{1, 2, ..., T\}$ for~$T \in \mathds{N}$.~$\log(\cdot)$ is the natural logarithm;~$g = O(f)$ means that for any~$z \in \text{Dom}(f) = \text{Dom}(g)$ one has~$f(z) \le C g(z)$ with~$C$ being a generic constant;~$g = \wt O(f)$ means the same but with~$C$ replaced by a poly-logarithmic factor in~$g$. 
We write~$\py F(x(y), y)$ for the partial gradient in~$y$ of~$F(x(y),y)$ as a function of~$y$; in other words,~$\py F(x(y), y) = \gy F(x,y)$ with~$x = x(y)$ substituted post-factum.
We shall introduce additional notation when the need arises.
\section{Building blocks and preliminaries}
\label{sec:toolbox}

Given a convex set~$Z$ in a Euclidean space~$\ZZ$ and a pair~$z,\zeta \in Z$, 
we define the {\em prox-mapping}
\begin{eq}
\label{eq:grad-map}
\prox_{z,Z}(\zeta) := \argmin_{z' \in Z} \lang \zeta, z'\rang +  \tfrac{1}{2} \|z' - z\|^2.
\end{eq}
In what follows, we assume~$\prox_{z,Z}(\zeta)$ to be computationally cheap. 
\odima{Note that in the unconstrained case with $Z = \ZZ$, one has~$\prox_{z,Z}(\zeta) = z - \zeta$, whereas in the (general) constrained case one has~$\prox_{z,Z}(\zeta) = \Pi_{Z} (z - \zeta)$.}
Furthermore, in what follows we use the notion of \textit{inexact first-order oracle} for a smooth convex function due to~\cite{devolder2014first}. 
\begin{definition}[$\delta$-inexact oracle]
\label{def:inexact-oracle}
Let~$f:Z\to\R$ be convex with~$L$-Lipschitz gradient. 
Pair~$[\wt f(\cdot), \wt \nabla f(\cdot)]$ is called~{\em inexact oracle for~$f$ with accuracy~$\delta \ge 0$} if for any pair of points~$z, z' \in Z$ one has
\begin{eq}
\label{eq:inexact-oracle}
0 \le f(z') - \wt f(z) - \langle \wt \nabla f(z), z' - z \rangle  \le \frac{L}{2}\|z'-z\|^2 + \delta.
\end{eq}
\end{definition}
Note that, unlike~\cite{devolder2014first}, we do not include~$L$ into the definition of inexact oracle. 

Next we present Nesterov's fast gradient method (FGM) for smooth convex optimization with inexact oracle (see~\cite{devolder2014first}) and a restart scheme for it. We use them in two scenarios: 
(a)~minimization of a strongly convex function on~$X$ with exact oracle;
(b)~maximization of a strongly concave function on~$Y$ with~$\delta$-inexact oracle.

\vspace{-0.2cm}
\hspace{-0.85cm}
\begin{minipage}{0.45\textwidth}
\begin{algorithm}[H]
\caption{Fast Gradient Method}
\label{alg:fgm}
\begin{algorithmic}[1]
\Function{\FGM}{$z_0,Z,\gamma,T,\wt\nabla f(\cdot)$}
\State $G_0 = 0$
\For{$t = 0, 1, ..., T-1$}
\State $u_t = \prox_{z_0,Z}\left(\gam G_t\right)$ 
\State $\tau_t = \frac{2(t+2)}{(t+1)(t+4)}$ 
\State $v_{t+1} = \tau_t u_t + (1-\tau_t)z_t$ 
\State $g_{t} = \frac{t+2}{2} \wt\nabla f(v_{t+1})$ 
\State $w_{t+1} = \prox_{u_t,Z}\left(\gam g_{t}\right)$
\State $z_{t+1} = \tau_t w_{t+1} + (1-\tau_t) z_t$
\State $G_{t+1} = G_t + g_{t}$
\EndFor \\
\Return $z_T$
\EndFunction
\end{algorithmic}
\end{algorithm}
\end{minipage}
\begin{minipage}{0.57\textwidth}
\vspace{-0.1cm}
\begin{algorithm}[H]
\caption{Restart Scheme for~FGM}
\label{alg:fgm-restart}
\begin{algorithmic}[1]
\Function{\FGMR}{$z^{0},Z,\gamma,T, S, \wt\nabla f(\cdot)$}
\For{$s \in [S]$}
\State $z^{s} = \; \FGM(z^{s-1},Z,\gamma,T,\wt\nabla f(\cdot))$
\EndFor\\
\Return $z^{S}$
\EndFunction
\end{algorithmic}
\end{algorithm}
\subsection{Fast gradient method with inexact oracle}
\label{sec:fgm-inexact}
%
%
%
%
%
%

Assume we are given initial point~$z_0 \in Z$, target number of iterations~$T$, stepsize~$\gam > 0$, and access to a~$\delta$-inexact (or, possibly, exact) oracle for function~$f: Z \to \R$ which satisfies the requirements in~\cref{def:inexact-oracle}. 

\end{minipage}
\vspace{0.05cm}

We will use a variant of fast gradient method with inexact oracle due to~\cite{devolder2014first}, given here as Algorithm~\ref{alg:fgm}, that performs~$T$ iterations and outputs approximate minimizer~$z_T$ of~$f$; each of these iterations reduces to a single call of~$\wt \nabla f(\cdot)$, two prox-mapping computations and a few entrywise vector operations. 
Note that the inexact oracle~$\wt\nabla f(\cdot)$ is passed as an input parameter (i.e., ``function handle''); this means that such an oracle must be implemented as an external procedure.

Assuming that the error of~$\wt \nabla f(\cdot)$ is small enough, 
the work~\citep{devolder2014first} ensures that the standard~$O(T^{-2})$ convergence of FGM is preserved.
Let us now rephrase their result.
\begin{theorem}[{\cite[Thm.~5 and~Eq.~(42)]{devolder2014first}}]
\label{th:devolder}
Running Algorithm~\ref{alg:fgm} runs with~$\gam = 1/L$ and~$\delta$-inexact oracle of~$f$ that is~$L$-smooth, convex, and minimized at~$z^*$ such that~$\|z_0 - z^* \| \le R$, ensures that
$
f(z_T) - f(z^*) \le {4 L R^2}/{T^2} + 2\delta T.
$
As a result, one has
\begin{eq}
\label{eq:fgm-oracle-accuracy-bound}
f(z_T) - f(z^*) \le \frac{5L R^2}{T^2} 
\;\; \text{whenever} \;\;
\delta \le \delta_T := \frac{LR^2}{2T^3}.
\end{eq}
\end{theorem}

When~$f$ is also~$\lambda$-strongly convex,~\cref{eq:fgm-oracle-accuracy-bound} allows to bound the distance to~$z^*$:
\begin{eq}
\label{eq:fgm-one-epoch-progress}
\|z_T - z^*\|^2 \le {10\kappa R^2}/{T^2},
\end{eq}
where~$\kappa = L/\lambda$ is the condition number. That is, we are guaranteed to get twice closer to the optimum after~$T = O(\sqrt{\kappa})$ iterations. 
Following~\cite{nesterov2013gradient}, we exploit this fact to obtain linear convergence via the simple restart scheme given in~\cref{alg:fgm-restart}, and derive the following result.

\begin{corollary}
\label{cor:fgm-restart-bound}
Run~\cref{alg:fgm-restart} with~$\gam = 1/L$, parameters~$T,S$ satisfying
\begin{eq}
\label{eq:fgm-restart-params}
T \ge \sqrt{40 \kappa}, \quad S \ge \log_{2}\left({3LR}/{\veps} \right)
\end{eq}
for some~$\veps > 0$, and~$\delta \le \delta_T$, cf.~\eqref{eq:fgm-oracle-accuracy-bound}. 
Then the final iterate~$z^{S}$ satisfies
\begin{eq}
\label{eq:fgm-approximation-bounds}
\begin{ald}
\|z^{S}- z^*\| \le \frac{\veps}{3L}, \quad  
f(z^{S}) - f(z^*) \le \frac{\veps^2}{18L}, \quad
\Dz(z^{S}, \nabla f(z^S), L) &\le \frac{\veps}{3}.
\end{ald}
\end{eq}
\end{corollary}

\begin{proof}
\odima{ 
By~\cref{eq:fgm-one-epoch-progress}, $T \ge \sqrt{40 \kappa}$ iterations in the first epoch ensure that~$\|z^{1} - z^*\| \le R/2$, i.e., we halve the initial distance to the optimum for the next epoch in which~\FGM{} is initialized at~$z^{1}$.
Repeating this process for~$s$ epochs, we have~$\|z^{s} - z^*\| \le 2^{-s} R$. 
In particular, after~$S \ge \log_{2}(3LR/\veps)$ epochs we arrive at the first bound in~\cref{eq:fgm-approximation-bounds}.
Now, arguing in a similar manner, but this time using~\cref{eq:fgm-oracle-accuracy-bound}, we have that
\[
\begin{ald}
f(z^{s}) - f(z^*) 
\le \frac{5L(2^{-s+1}R)^2}{T^2}
\le \frac{20 LR^2}{4^s T^2}  \le \frac{LR^2}{2^{2s+1}}, 
\end{ald}
\]
where in the first step we combined~\cref{eq:fgm-oracle-accuracy-bound} with with the bound~$\|z^{s-1} - z^*\| \le 2^{-s+1} R$, and in the end we used~$\kappa \ge 1$.}
Plugging in~$2^{2S+1} = 18 L^2R^2/\veps^2$, we verify the second inequality in~\cref{eq:fgm-approximation-bounds}.
Finally, for the last inequality in~\cref{eq:fgm-approximation-bounds}, we first observe that, due to the smoothness of~$f$, it holds that
$
f(z) - f(z^S) \le \lang \nabla f(z^S), z - z^S \rang  + \frac{L}{2}\|z-z^S\|^2.
$
Thus, one has~$f(z^*) - f(z^S) \le \min_{z \in Z}[\lang \nabla f(z^S), z - z^S \rang  + \frac{L}{2}\|z-z^S\|^2]$, whence
\small
\begin{eq}
\label{eq:-fgm-near-stationarity-argument}
\begin{aligned}
&\Dz^2(z^{S}, \nabla f(z^S), L) 
= 2L \max_{z \in Z} \left[-\langle \nabla f(z^S), z - z^S \rangle - \tfrac{L}{2}\|z-z^S\|^2 \right] \\
= &-2L \min_{z \in Z} \left[\langle \nabla f(z^S), z - z^S \rangle  + \tfrac{L}{2}\|z-z^S\|^2 \right] \le -2L [f(z^*) - f(z^S)] \le {\veps^2}/{9},
\end{aligned}
\end{eq}
\normalsize
where the final inequality uses the second part of~\cref{eq:fgm-approximation-bounds} proved earlier.
\end{proof}

When~\cref{alg:fgm-restart} is used for minimization in~$x$, the complexity factor~$\Tx$ is parametrized by the \odima{upper bound on the initial objective gap~$\Gap_f [\ge f(z^0) - f(z^*)]$} rather than~$R$, 
and the \textit{exact} oracle~$\nabla f(\cdot)$ is available (function value is not used).
\odima{Note that, by strong convexity, such a bound also implies a bound on the initial distance to the optimum, namely~$R^2 = 2\kappa \Gap_f / L$, 
and we arrive at the following result.} 
\begin{corollary}
\label{cor:unconstrained-case}
\odima{Assume that~$f(z^0) - f(z^*) \le \Gap_f$.}
Run~\cref{alg:fgm-restart} with~$\delta = 0$, 
\begin{eq}
\label{eq:restarts-from-gap}
S \ge \frac{1}{2}\log_2\left({18\kappa L \Gap_f}/{\veps^2}\right),
\end{eq}
and other parameters set as in~\cref{cor:fgm-restart-bound}. Then the bounds in~\cref{eq:fgm-approximation-bounds} remain valid.
\end{corollary}

\subsection{Proximal point operator and its implementation via FGM}
\label{sec:ppm-via-fgm}
Next we briefly review the proximal point method, which forms the backbone of our approach, in the context of searching for stationary points of nonconvex functions.
Then we show how the iterations of this method can be approximated by using~\cref{alg:fgm-restart}. 

Given a convex set~$X$ and~$\phi: X \to \R$ with~$L$-Lipschitz gradient,
the {\em proximal point operator of~$\phi$ on~$X$ with stepsize~$0 < \gam < 1/L$} is defined by
\begin{eq}
\label{def:prox-map}
x \mapsto x^+_{\gam\phi, X}(x) := \argmin_{x' \in X} \left[ \phi (x') + \frac{1}{2 \gam} \|x' - x\|^2 \right].
\end{eq}
Denoting~$x^+ = x^+_{\gam\phi,X}(x)$ for brevity, the first-order optimality condition in~\cref{def:prox-map} writes
\begin{eq}
\label{eq:ppm-foo}
\big\langle \nabla\phi(x^+) + \tfrac{1}{\gamma} (x^+ - x), x'-x^+ \big\rangle \ge 0, \quad \forall x' \in X.
\end{eq}
Note that this reduces to the ``implicit gradient descent'' update~$x^{+} = x - \gam \nabla \phi(x^{+})$ in the unconstrained case.
For large stepsize, computing the proximal operator at a point might be as hard as minimizing~$\phi$.  
However, with sufficient regularization, namely when~$\gam = c/L$ for~$0 < c \le 1/2$, the task becomes easy, since the objective in~\cref{def:prox-map} is strongly convex and well-conditioned, with 
$\kappa = (1+c)/(1-c) \le 3.$
On the other hand, with such stepsize the \textit{proximal point method}, as given by
\begin{eq}
\label{eq:prox-update}
x_{t} = x^+_{\gam \phi, X}(x_{t-1}),
\end{eq}
attains the optimal rate~$O(1/\sqrt{T})$ of minimizing the stationarity measure~$\Dx$ (cf.~\cref{def:stationarity}).
Indeed, from~\eqref{def:prox-map} with~$\gam = c/L$ we get
\begin{eq}
\label{eq:prox-point-step-progress}
\phi(x^+) + \frac{L}{2c}\|x^{+} - x\|^2 \le \phi(x).
\end{eq}
Iterating this~$T$ times according to~\eqref{eq:prox-update} results in
\begin{eq}
\label{eq:prox-point-argument}
\min_{t \in [T]} \|x_t - x_{t-1}\|^2
\le \frac{1}{T} \sum_{t \in [T]} \| x_t - x_{t-1} \|^2 \le \frac{2c\Gap}{LT},
\end{eq}
where~$\Delta = \phi(x_0) - \textstyle\min_{x \in X}\phi(x)$ is the initial gap.
On the other hand, 
\begin{eq}
\label{eq:stationarity-by-closeness}
\begin{ald}
\Dx^2(x^+,\nabla \phi(x^+), L) 
&\equiv 2L \max_{x' \in X} \left[- \lang \nabla \phi(x^+), x' - x^+ \rang - \frac{L\|x' - x^+\|^2}{2} \right] \\
&\le 2L^2 \max_{x' \in X} \left[\frac{1}{c} \lang x^+ - x, x' - x^+ \rang - \frac{\|x'-x^+\|^2}{2} \right]
\le \frac{L^2 \|x^+ - x\|^2}{c^2},
\end{ald}
\end{eq}
where we first used the first-order optimality condition~\cref{eq:ppm-foo} and then Young's inequality; note that the last inequality becomes tight when~$X = \XX$. 
Combining~\cref{eq:prox-update},~\cref{eq:prox-point-argument} and~\cref{eq:stationarity-by-closeness}, we arrive at
\begin{eq}
\label{eq:intro-ppm-rate}
\min_{t \in [T]} \Dx(x_t,\nabla \phi(x_t), L)  \le \sqrt{\frac{2L\Gap}{cT}},
\end{eq}
i.e., the iteration complexity  
$T(\veps) = O\left({L \Gap}/{\veps^2}\right)$
of minimizing the measure~$\Dx$, which is optimal in the unconstrained case~\cite{carmon2017lower}.

Of course, the above argument would be useless if~\cref{eq:prox-point-argument} or~\cref{eq:stationarity-by-closeness} were not tolerant to errors when computing~$x^+_{\gam\phi, X}(x)$, i.e., when minimizing the regularized function
\begin{eq}
\label{eq:ppm-regularized-function}
\phi_{L,x}(\cdot) := \phi(\cdot) + L\|\cdot - x\|^2.
\end{eq}
(Here we fixed~$c = 1/2$ for simplicity, i.e.,~$\gam = 1/(2L)$, cf.~\cref{def:prox-map}.)
\odima{We shall now verify such error-tolerance for~\cref{eq:prox-point-argument},~\cref{eq:stationarity-by-closeness} and~\cref{eq:intro-ppm-rate} as a result.}
Indeed, let~$\wt x^+ \in X$ satisfy
\begin{eq}
\label{eq:ppm-objective-approx}
\phi_{L,x}(\wt x^+) \le \phi_{L,x}(x^+) + \frac{\veps^2}{24L}
\end{eq}
for given~$x$, where~$x^+ = x^+_{\phi/2L, X}(x)$ is the true minimizer,~$\veps$ the desired accuracy, 
and the constant~$1/24$ will be \odima{convenient in further calculations.}
Consider the counterpart of~\cref{eq:prox-update}, i.e., the sequence~$\wt x_t = \wt x^+_{\phi/2L,X}(\wt x_{t-1})$ obeing~\eqref{eq:ppm-objective-approx} at each step.
Using~\cref{eq:prox-point-step-progress} and proceeding as when deriving~\cref{eq:prox-point-argument}, we obtain
the following counterpart of~\cref{eq:prox-point-argument}:
\begin{eq}
\label{eq:prox-point-argument-approx}
\min_{t \in [T]} \|\wt x_t - \wt x_{t-1}\|^2
\le \frac{1}{T} \sum_{t \in [T]} \| \wt x_t - \wt x_{t-1} \|^2 \le \frac{\Gap}{LT} + \frac{\veps^2}{24L^2},
\end{eq}
\odima{thus showing the desired error-tolerance for~\cref{eq:prox-point-argument}.}
Moreover, assume now that~$\wt x^+$, in addition to~\cref{eq:ppm-objective-approx}, admits the matching guarantee for the stationarity measure, that is
\begin{eq}
\label{eq:ppm-stationarity-approx}
\Dx(\wt x^+, \nabla \phi_{L,x}(\wt x^+), L) \le {\veps}/{2}.
\end{eq} 
Then \odima{we sequentially obtain}
\begin{eq}
\hspace{-0.13cm}
\begin{ald}
&\Dx^2(\wt x^+, \nabla \phi(\wt x^+), L) 
\equiv 2L \max_{x' \in X} \left[- \lang \nabla \phi(\wt x^+), x' - \wt x^+ \rang - \tfrac{L}{2} \|x' - \wt x^+\|^2 \right] \\
\le & 2L \max_{x' \in X} \left[- \lang \nabla \phi(\wt x^+) + 2L (\wt x^+ - x), x' - \wt x^+ \rang - \tfrac{L}{4} \|x' - \wt x^+\|^2 \right] \\
&\quad\quad\quad+ 2L \max_{x' \in X} \left[\lang 2L (\wt x^+ - x), x' - \wt x^+ \rang - \tfrac{L}{4} \|x' - \wt x^+\|^2 \right]\\
= & 2 \Dx^2(\wt x^+, \nabla \phi_{L,x}(\wt x^+), L/2) + 8L^2 \|\wt x^+ - x\|^2 \\
\le & 2 \Dx^2(\wt x^+, \nabla \phi_{L,x}(\wt x^+), L) + 8L^2 \|\wt x^+ - x\|^2 
= {\veps^2}/{2} + 8L^2 \|\wt x^+ - x\|^2;
\end{ald}
\label{eq:stationarity-by-closeness-approx}
\end{eq}
here we first used the explicit form of~$\nabla\phi_{L,x}$, then estimated the additional term via Young's inequality, and finally used that~$\Dx(x, \xi, L)$ is non-decreasing in~$L$, as follows from the proximal Polyak-Lojasiewicz lemma~(\cite[Lem.~1]{karimi2016linear}).\footnote{Namely, we apply~\cite[Lemma~1]{karimi2016linear} using the indicator of~$X$ as the proximal function~$g(x)$ there.}
\odima{Thus, we have just verified the required error-tolerance for~\cref{eq:stationarity-by-closeness}.}
Finally, by recalling~\cref{eq:prox-point-argument-approx} we arrive at
\begin{eq}
\label{eq:stationarity-guarantee-approx}
\min_{t \in [T]} \Dx(\wt x_t,\nabla \phi(\wt x_t), L)  \le \sqrt{\frac{8L\Gap}{T} + \frac{5\veps^2}{6}} \le 3\sqrt{\frac{L\Gap}{T}} + \veps,
\end{eq}
which results in the same complexity~$T(\veps) = O\left({L \Gap}/{\veps^2}\right)$
as for the exact updates~\cref{eq:prox-update}.\\

It remains to notice that the point~$\bx^+$ satisfying~\cref{eq:ppm-objective-approx} and~\cref{eq:ppm-stationarity-approx} can be obtained by running FGM with restarts (\cref{alg:fgm-restart}) with a near-constant total number of oracle calls, since the function~$\phi_{L,x}$ minimized in~\cref{def:prox-map} is~$3L$-smooth and~$L$-strongly-convex. 
Namely, combining~\cref{cor:fgm-restart-bound} and~\cref{cor:unconstrained-case}, we obtain the following.
\begin{proposition}[Implementation of proximal point operator via FGM]
\label{prop:ppm-via-fgm}
Given some~$x \in X$, let~$\phi: X \to \R$ have~$L$-Lipschitz gradient, and let~$x^+ = x^+_{\phi/(2L),X}(x)$ be the minimizer of~$\phi_{L,x},$ cf.~\cref{eq:ppm-regularized-function}.
Let~$\wt x^+$ be the output of~\cref{alg:fgm-restart} run with exact oracle~$\nabla \phi_{L,x}(\cdot)$,~$z_0 = x,$~$Z = X$, and parameters
\begin{eq}
\label{eq:ppm-via-fgm-params}
T = 11, \quad \gam = \frac{1}{3L},  \;\; \text{and} \;\; S \ge \frac{1}{2}\log_2\left(\frac{72 L\Delta_{L,x}}{\veps^2} \right),
\end{eq}
where~$\Delta_{L,x} := \phi(x) - \textstyle\min_{x'}\phi_{L,x}(x')$. 
Then 
\begin{eq}
\label{eq:ppm-via-fgm-regularized-bounds}
\begin{ald}
\| \bx^+ - x^+ \| \le \frac{\veps}{6L}, \quad
\Dx(\bx^+, \nabla \phi_{L,x}(\bx^+), L) \le \frac{\veps}{2}, \quad
\phi_{L,x}(\bx^+) - \phi_{L,x}(x^+) \le \frac{\veps^2}{24L}.
\end{ald}
\end{eq}
\end{proposition}

\begin{proof}
Note that~$\phi_{L,x}(\cdot)$ is~$3L$-smooth, has condition number~$\kappa \le 3$, and is~$\Gap_{L,x}$-suboptimal at~$x$.
Hence,~\cref{alg:fgm-restart} run with~$T = 11 > \sqrt{40 \kappa}$ and 
\[
S \ge \frac{1}{2} \log_2 \left( \frac{72 L \Gap_{L,x}}{\veps^2}\right) \ge \frac{1}{2} \log_2 \left( \frac{18 \kappa (3L) \Gap_{L,x}}{(3\veps/2)^2}\right),
\]
cf.~\cref{eq:restarts-from-gap}, outputs a point for which~\cref{eq:fgm-approximation-bounds} holds with the following replacements:
\[
z^S \mapsto \bx^+, \; z^* \mapsto x^+, \; f(\cdot) \mapsto \phi_{L,x}(\cdot), \; L \mapsto 3L, \; \veps \mapsto \frac{3\veps}{2}.
\] 
Using that~$\Dx(x,\xi,L) \le \Dx(x,\xi,3L)$, we verify all three inequalities in~\cref{eq:ppm-via-fgm-regularized-bounds}.
\end{proof}


%

\section{Algorithm and main result}
\label{sec:algorithm}

\odima{In order to better convey the ideas behind our approach, we shall present it in a similar manner as in~\cref{sec:ppm-via-fgm}.
Namely, we shall first present the ``conceptual'' algorithm with {\em exact} proximal-point type updates,
and then show how to approximate these updates, which shall result in our final algorithm.}

\subsection{Conceptual algorithm: primal-dual proximal point iteration}
\label{sec:algo-outline}

First, following~\cite{nesterov2012howto}, we reduce the problem of finding~$(\ex, \ey)$-FNE in~\cref{opt:min-max} to the problem of finding approximate FNE of the regularized function
\begin{eq}
\label{eq:regularized-surrogate}
\bF(x,y) := F(x,y) - \frac{\ey}{2\Ry} \|y - \by\|^2.
\end{eq}
This function has a unique maximizer for any~$x \in X$ as it is~$\ey/\Ry$-strongly concave. This strong concavity will help us obtain faster algorithms for finding $(\ex, \ey)$-FNE when applying standard accelerated procedures.

The crux of our approach is to run a version of primal-dual proximal-point method, choosing the next iterate~$(x_{t}, y_{t})$ as an approximate optimal solution to the convex-concave saddle-point problem (with unique exact solution):
\begin{eq}
\label{eq:minmax-conceptual-updates}
\begin{ald}
\min_{x \in X} \max_{y \in Y}  \left[ \bF_{t}(x,y) := \bF(x,y) + \Lxx \|x - x_{t-1}\|^2 \right].
\end{ald}
\end{eq}
To illustrate this idea, let us consider the idealized iterates~$(\wh x_t, \wh y_t)$ corresponding to the exact saddle point in~\cref{eq:minmax-conceptual-updates}, which exists and is unique by Sion's minimax theorem~\cite{sion1958}.
By definition, we have
\begin{eq}
\label{eq:idea-saddle-point}
\begin{ald}
\bF_t(\wh x_{t}, \wh y_{t+1})  \le \bF_t(\wh x_{t}, \wh y_{t}) \le \bF_t(\wh x_{t-1}, \wh y_{t}).
\end{ald}
\end{eq}
Using the expression for~$\bF_t(x,y)$ in~\cref{eq:minmax-conceptual-updates}, the \odima{right-hand inequality} in~\cref{eq:idea-saddle-point} reads
\begin{eq}
\label{eq:ppm-primal-dual-simple-half}
\bF(\wh x_{t}, \wh y_{t}) + \Lxx \|\wh x_{t} - \wh x_{t-1}\|^2 \le \bF(\wh x_{t-1}, \wh y_{t}).
\end{eq}
\odima{Meanwhile, the first inequality in~\cref{eq:idea-saddle-point} implies that~$\bF(\wh x_{t}, \wh y_{t}) \ge \bF(\wh x_{t}, \wh y_{t+1})$.}
Thus we arrive at
\begin{eq}
\label{eq:ppm-primal-dual-simple}
\bF(\wh x_{t}, \wh y_{t+1}) + \Lxx \|\wh x_{t} - \wh x_{t-1}\|^2 \le \bF(\wh x_{t-1}, \wh y_{t}).
\end{eq}
\odima{
The point here is that, unlike~\cref{eq:ppm-primal-dual-simple-half}, relation~\cref{eq:ppm-primal-dual-simple} can now be {iterated}, as the index gets shifted in the left-hand side for {\em both} variables.
Iterating~\cref{eq:ppm-primal-dual-simple} results in a similar argument as in~\cref{eq:prox-point-step-progress}-\cref{eq:intro-ppm-rate} and gives the~$\Tx$ complexity factor.
More precisely, applying~\cref{eq:ppm-primal-dual-simple} for~$t \in [T-1]$ and~\cref{eq:ppm-primal-dual-simple-half} at~$t = T$, we arrive at an analogue of~\cref{eq:prox-point-argument}:} 
\[
\min_{t \in [T]} \|\wh x_{t} - \wh x_{t-1}\|^2
\le \frac{1}{T} \sum_{t \in [T]} \|\wh x_{t} - \wh x_{t-1}\|^2 \odima{\le} \frac{\bF(\wh x_0, \wh y_1) - \bF(\wh x_T, \wh y_T)}{\Lxx T}.
\]
Now observe that we can relate~$\bF(\wh x_0, \wh y_1) - \bF(\wh x_T, \wh y_T)$ to~$\Gap = \vphi(\wh x_0) - \textstyle \min_{x \in X} \vphi(x)$: 
\begin{eq}
\label{eq:gap-estimate}
\begin{ald}
\bF(\wh x_0, \wh y_1) &\le F(\wh x_0, \wh y_1) \le \textstyle\max_{y \in Y} F(\wh x_0, y) = \vphi(\wh x_0), \\
\bF(\wh x_T, \wh y_T) 
&\ge  \max_{y \in Y} F(\wh x_T,y) - \frac{\ey}{2\Ry} \max_{y' \in Y} \|y' - \by\|^2  
\ge \min_{x \in X} \vphi(x) - 2\ey \Ry.
\end{ald}
\end{eq}
Thus we can guarantee the existence of~$\tau \in [T]$ for which
$
\|\wh x_{\tau} - \wh x_{\tau-1} \|^2  \odima{\le} \frac{\Gap + 2\ey\Ry}{\Lxx T},
$
mimicking~\cref{eq:prox-point-argument} up to~$O(\ey)$ additive error. 
Now we can proceed as in~\cref{eq:stationarity-by-closeness}, using the primal optimality condition in~\cref{eq:minmax-conceptual-updates} and~$\gx \bF(x,y) \equiv \gx F(x,y)$. This results in
\[
\Dx(\wh x_\tau, \gx F(\wh x_\tau, \wh y_{\tau}), \Lxx) \le  2\sqrt{\frac{\Lxx(\Gap + 2\ey\Ry)}{T}},
\]
corresponding to~$O(\Tx)$ iterations~\cref{eq:complexity-factors} to ensure~$\Dx(\wh x_\tau, \gx F(\wh x_\tau, \wh y_{\tau}), \Lxx) \le \ex$. 
Meanwhile, we remain near-stationary in~$y$: indeed, for any iteration~$t \in [\Tx]$ we have that
\small
\[
\begin{ald}
&\Dy^2(\wh y_t, - \gy F(\wh x_t, \wh y_t), \Lyy) 
= 2\Lyy \max_{y \in Y} \left[\lang \gy F(\wh x_{t}, \wh y_{t}), y - \wh y_{t} \rang  - \frac{\Lyy}{2}\|y-\wh y_t\|^2 \right] \\
\le & 2\Lyy \max_{y \in Y} \left[\lang \frac{\ey}{\Ry}(\wh y_{t}-\by), y - \wh y_{t} \rang - \frac{\Lyy}{2}\|y-\wh y_t\|^2\right] \\
&\quad\quad\quad\quad 
	+ 2\Lyy \max_{y \in Y} \left[\lang \gy F(\wh x_{t}, \wh y_{t}) - \frac{\ey}{\Ry}(\wh y_{t} - \by), y - \wh y_{t} \rang \right]\\
\le & 2\Lyy \max_{y \in Y} \left[\lang \frac{\ey}{\Ry}(\wh y_{t}-\by), y - \wh y_{t} \rang - \frac{\Lyy}{2}\|y-\wh y_t\|^2\right] = \frac{\ey^2}{\Ry^2} \|\wh y_t - \by\|^2 \le {4\ey^2},
\end{ald}
\]
\normalsize
where in the second inequality we used the dual optimality condition for~\cref{eq:minmax-conceptual-updates}, and then used Young's inequality.
Thus,~$(\wh x_\tau, \wh y_\tau)$ is an~$(\ex,O(\ey))$-FNE. 

So far  we assumed the update~\eqref{eq:minmax-conceptual-updates} can be done exactly and analyzed the iteration complexity of the resulting idealized procedure.  Next we show how to approximate~\eqref{eq:minmax-conceptual-updates} via~\cref{alg:fgm-restart}, leading to our final algorithm and its efficiency estimate.

\subsection{Implementation of conceptual algorithm}
\label{sec:implementation}

As in the case of the usual proximal point method, the update stemming from the auxilliary min-max problem in~\cref{eq:minmax-conceptual-updates} cannot be performed exactly. 
To address this problem, we extend the approach described in~\cref{sec:ppm-via-fgm} and approximately solve the (primal) minimization problem in~\cref{eq:minmax-conceptual-updates} up to~$O(\ex)$ accuracy in the~$\Dx$-measure via~\cref{alg:fgm-restart} (cf.~\cref{prop:ppm-via-fgm}).
The key challenge here is that the function to minimize in~\cref{eq:minmax-conceptual-updates} stems from the nested maximization problem, hence neither it nor its gradient can be computed exactly. 
Instead, we provide \textit{inexact oracle} for this function through the following steps.

\vspace{-0.2cm}
\begin{algorithm}[H]
\caption{Solve Regularized Dual Problem}
\label{alg:grad-approx}
\begin{algorithmic}[1]
\Function{\AppGrad}{$y,x_{t-1},\by,\gamx,\lamy,T,S$} 
\State \hspace{-0.3cm} $\wt x_t(y) = \, \FGMR(x_{t-1},X,\tfrac{2}{3}\gamx,T,S, 
\gx F(\cdot,y) - \tfrac{1}{\gamx}(\cdot - x_{t-1}))$
\State \hspace{-0.3cm} $\wt \nabla \psi_t(y) = \gy F(\wt x_t(y),y) - \lamy(y-\by)$\\
\Return $\wt x_t(y), \wt \nabla \psi_t(y)$
\EndFunction
\end{algorithmic}
\end{algorithm}\vspace{-0.4cm}

First, given the current primal iterate~$x_{t-1}$, consider the minimization problem corresponding to the dual function of~\cref{eq:minmax-conceptual-updates} evaluated at some fixed~$y \in Y$:
\[
\psi_t(y) := \min_{x \in X} \left[ \bF_t(x,y)  := \bF(x,y) + \Lxx \|x - x_{t-1}\|^2 \right].
\] 
Solving this minimization problem for fixed~$y\in Y$ by running~\cref{alg:fgm-restart} with exact oracle~$\gx F(\cdot,y) + 2\Lxx(\cdot - x_{t-1})$, we obtain approximation~$\wt x_t(y)$ of the exact minimizer~$\wh x_t(y)$.
As~$F_t(\cdot,y)$ is well-conditioned, it only takes a logarithmic number of oracle calls to ensure a very small (inversely polynomial in the problem parameters) error of approximating~$\wh x_t(y)$. On the other hand, a version of Danskin's theorem (\cite[Lem.~24]{nouiehed2019solving}) guarantees that the gradient of~$\psi_t(y)$, given by
\begin{eq}
\label{eq:psi-gradient}
\nabla \psi_t(y) \equiv \py \bF(\wh x_t(y), y),
\end{eq}
is~$O(\Lyy^+)$-Lipschitz.
Hence,~$\wt x_t(y)$ provides a~$\delta$-inexact oracle for~$\psi_t(y)$:\vspace{-0.15cm}
\begin{eq}
\label{eq:psi-inexact-oracle}
\wt \psi_t(y) := \bF(\wt x_t(y), y),  \quad 
\wt \nabla \psi_t(y) := \py \bF(\wt x_t(y),y),
\vspace{-0.1cm}
\end{eq} 
cf.~\cref{def:inexact-oracle}, where the accuracy parameter~$\delta$ can be arbitrarily chosen.
For convenience, we outline the subroutine that returns~$\wt x_t(y)$ and the approximate dual gradient~$\wt \nabla \psi_t(y)$ in~\cref{alg:grad-approx}.
Now, observe that we can switch the order of~$\min$ and~$\max$ in~\cref{eq:minmax-conceptual-updates}, recasting it as
$
y_t = \arg\max_{y \in Y} \psi_t(y),$ and
$x_t = \wh x_t(y_t).$
Naturally, we replace those with the approximate updates given by
\begin{eq}
\label{eq:minmax-approx-updates}
\begin{ald}
y_t &\approx \arg\max_{y \in Y} \psi_t(y), \quad 
x_t = \wt x_t(y_t),
\end{ald}
\vspace{-0.2cm}
\end{eq}
maximizing~$\psi_t(y)$ by running Algorithm~\ref{alg:fgm-restart} with inexact gradient~$-\wt \nabla \psi_t(y)$ defined in~\cref{eq:psi-inexact-oracle}, {\em and without using~$\wt \psi_t(y)$}. Since~$\psi_t(y)$ is~$\Lyy^+$-smooth and~$(\ey/\Ry)$-strongly concave, in~$O(\Ty)$ calls of the inexact oracle~$-\wt \nabla \psi_t(\cdot)$~\cref{alg:fgm-restart} finds~$O(\ey)$-approximate maximizer~$y_t$ of~$\psi_t$, ensuring that\vspace{-0.2cm}
\[
\vspace{-0.1cm}
\Dy(y_t, -\nabla \psi_t(y_t), \Lyy^+) \le  \frac{\ey}{3},
\quad\quad  \max_{y \in Y} \psi_t(y) - \psi_t(y_t) \le \frac{\ey^2}{18\Lyy^+}.
\]
Combining the first of these inequalities with~\cref{eq:psi-gradient} and recalling that~$\wt x_t(y_t) \approx \wh x_t(y_t)$ with very high accuracy, we ensure that~$(x_t,y_t)$ obtained via~\cref{eq:minmax-approx-updates} is~$O(\ey)$-stationary in~$y$ (in the sense of~\cref{def:stationarity}). 
As this must be repeated for~$t \in [\Tx]$, we recover the first term in~\cref{eq:complexity-simple}. 
On the other hand, the second inequality leads to the extra~$O(\ey^2/\Lyy^+)$ error in the saddle point relation~\cref{eq:idea-saddle-point}, whereas, as we know from~\cref{prop:ppm-via-fgm}, this error must be~$O(\ex^2/\Lxx)$ in order to preserve the argument in~\cref{sec:algo-outline}. 
This is easy to fix: it suffices to perform a logarithmic in~$\Tx$ number of {additional} restarts when maximizing~$\psi_t(y)$ (cf.~\cref{eq:ppm-via-fgm-params}).
Thus, the argument in~\cref{sec:algo-outline} remains valid, and we find~$(\ex,O(\ey))$-FNE in~\cref{opt:min-max} in~$\wt O(\Tx \Ty)$ gradient computations and projections.
The resulting algorithm, our main practical contribution, is given in~\cref{alg:fsp}. 
%
\vspace{-0.2cm}
\begin{algorithm}[H]
\caption{FNE Search in Nonconvex-Concave Smooth Min-Max Problem}
\label{alg:fsp}
\begin{algorithmic}[1]
\Require~$\nabla F(\cdot,\cdot)$,~$Y$,~$x_0$,~$\by \in Y$,~$\bTx$,~$\bTy$,~$\Sy$,~$\gamx$,~$\gamy$,~$\lamy$,~$\To$,~$\So$ 
\For{$t \in [\bTx]$} \Comment{Using~\cref{alg:fgm-restart,alg:grad-approx} as subroutines}
\State $y_{t} = \; \FGMR (\by, Y, \gamy, \bTy, \Sy, -\wt\nabla\psi_t(\cdot))$ 
\label{line:dual-update}
\State \quad \quad with~$\wt\nabla\psi_t(y)$ returned by~$\AppGrad(y,x_{t-1},\by,\gamx,\lamy,\To,\So)$
\label{line:call-approx-grad}
\State $x_{t} = \wt x_t(y_t)$ returned by~$\AppGrad(y_t,x_{t-1},\by,\gamx,\lamy,\To,\So)$ 
\label{line:prim-update}
\EndFor \\
\Return $(x_{\tau}, y_{\tau})$ \text{with}~$\tau \in \textstyle\Argmin_{t \in [\bTx]} \|\gx F(x_t, y_t)\|$
\end{algorithmic}
\end{algorithm}\vspace{-0.5cm}

\subsection{Convergence guarantee for~\cref{alg:fsp}}
\label{sec:upper-bound}
We state our main result. 

\begin{theorem}
\label{th:upper-bound}
Define $\lamy := \frac{\bey}{\Ry}, \quad \Theta := \Lyy \Ry^2, \quad \Theta^+ := \Lyy^+ \Ry^2,$ and
\begin{align}
\label{eq:delta-value}
\delta := \min\left[8\bey\Ry, \; \frac{\Theta}{2\bTy^{3}}, \; \sqrt{\frac{\Gap (\Theta^+ - \Theta)}{\bTx^{\vphantom 2} \bTy^2}}\right].
\end{align}
Let us run~\cref{alg:fsp} with
\vspace{-0.3cm}
\begin{equation}
\gamx = \frac{1}{2\Lxx}, \;\; \gamy = \frac{1}{\Lyy^+ + \lamy},
\label{eq:stepsize-params}
\end{equation}
\vspace{-0.4cm}
\begin{equation}
\bTx \ge \frac{10\Lxx(\Delta+2\bey\Ry)}{\ex^2}, \;\;
\bTy \ge \sqrt{\frac{40(\Lyy^+ + \lamy)}{\lamy}}, \;\;
\Sy \ge 2\log_2 \left( \max\left[\bTy, \frac{\Theta^+}{\delta}\right] \right),
\label{eq:outer-loop-params}
\end{equation}
\vspace{-0.4cm}
\begin{equation}
\To = 11, \quad 
\So \ge \frac{1}{2}\log_2\left( 72 (3\Gap + 2\Theta + 6\bey \Ry) \left[ \frac{\Lxx}{\ex^2} + \frac{2\Theta^+}{\delta^2} + \frac{1}{12\delta} \right]\right).
\label{eq:inner-loop-params}
\end{equation}
Its output is~$(2\ex,5\ey)$-FNE in the problem~\cref{opt:min-max}, in the sense of~\cref{def:stationarity}, in
$\left\lceil \To \So  \Sy \bTx \bTy \right\rceil$
computations of~$\nabla F(x,y)$ and twice that many projections onto~$X$~and~$Y$. 
\end{theorem}

We emphasize that our criterion of approximate FNE (cf.~\cref{def:stationarity}) is stronger than the criterion based on the proximal gradient: the obtained point~$(\wh x, \wh y)$ also satisfies
\[
\Lxx \left\| \wh x - \proj_{X}\left(\wh x - \tfrac{1}{\Lxx} \gx F(\wh x, \wh y) \right) \right\| \le 2\ex, 
\quad
\Lyy \left\| \wh y - \proj_{Y}\left(\wh y + \tfrac{1}{\Lyy} \gy F(\wh x, \wh y) \right) \right\|  \le 5\ey,
\] 
cf.~Remark~\ref{rem:criteria-comparison}. 
On the other hand, the converse is not true: the above guarantee is not sufficient to conclude that the point is~$(\ex,\ey)$-FNE in the sense of~\cref{def:stationarity}. 

\begin{rem}[Nonconvex-strongly-concave setup]
\label{rem:strongly-concave}
From~\cref{th:upper-bound} we see that~\cref{alg:fsp} can also be used when the objective~$F(x,y)$ is~$\lamy$-strongly concave in~$y$ with general~$\lamy$, leading to the complexity estimate~$\wt O(\Tx (\kappa_{\y}^+)^{1/2})$, where~$\kappa_{\y}^+ = \Lyy^+/\lamy$ is the condition number of the dual function \odima{in~\cref{eq:minmax-conceptual-updates}.} 
This matches the best known rate (see,~e.g.,~\cite{lin2020near}).
To this end, it suffices to run the algorithm with parameters set as in the premise of~\cref{th:upper-bound}, but fixing a prescribed value for~$\lamy$. 
\end{rem}

\odima{
\begin{rem}[Unknown~$\Gap$ and adaptive termination criterion]
As per~\cref{th:upper-bound}, the value~$\Gap$ enters the prescribed setup of parameters for~\cref{alg:fsp},  in three places: in the expression for the number~$\bTx$ of iterations in the outer loop and under the logarithms in~$S^o$ and~$S_y$ through~$\delta$, cf.~\cref{eq:delta-value}. 
In practice,~$\Gap$ is usually unknown, but this does not pose a problem. 
Indeed, in the case of logarithmic dependencies (in $S^o$ and~$S_y$), we can use, instead of~$\Delta$, a very crude upper bound (e.g., we always have~$\Delta \le 2\Lxx^{\vphantom 2}\Rx^2$ whenever~$X$ is contained in the Euclidean ball with radius~$\Rx$). 
As for~$\bTx$, observe that, when actually running~\cref{alg:fsp}, one does not have to fix in advance the number of outer loop iterations. 
Instead, one can run an infinite loop and check the stopping criterion~$\Dx(x_t, \gx F(x_t, y_t),\Lxx) \le 2\ex$ after each iteration, which amounts to computing a prox-mapping for~$X$. 
This is a valid stopping criterion: as follows from the proof of~\cref{th:upper-bound}, the complementary condition~$\Dy(y_t, -\gy F(x_t, y_t),\Lyy) \le 5\ey$ is maintained at each~$t$. 
To this end,~\cref{th:upper-bound} guarantees the termination of \cref{alg:fsp} after at most~$\bTx$ outer loop iterations. 
\end{rem}
}

\subsection{Proof of~\cref{th:upper-bound}}
\label{sec:upper-bound-proof}

We use the notation introduced in~\cref{sec:algo-outline}--\ref{sec:implementation} and refer to the arguments presented there if needed.

\proofstep{1}.
\odima{Given a primal iterate~$x_{t-1}$, let us define the following auxiliary functions:
\vspace{-0.2cm}
\[
\begin{aligned}
F_t(x,y) &:= F(x,y) + \Lxx \|x - x_{t-1}\|^2, \\
\bF(x,y) &:= F(x,y) - \tfrac{1}{2} \lamy \|y - \by\|^2, \\
\bF_t(x,y) &:= F_t(x,y) - \tfrac{1}{2} \lamy \|y - \by\|^2 \quad [= \bF(x,y)  + \Lxx \|x - x_{t-1}\|^2] \;. \\
\end{aligned}
\]
}
Consider first the ``idealized'' update from the primal iterate~$x_{t-1}$, as given by
\begin{eq}
\label{eq:conceptual-updates-proof}
\begin{ald}
y_t &= \arg\max_{y \in Y} \psi_t(y), \quad
x_t = \wh x_t(y_t).
\end{ald}
\end{eq}
Here,~$\psi_t(y)$ and~$\wh x_t(y)$ are defined as
\begin{eq}
\begin{ald}
\label{eq:conceptual-updates-proof-dual}
\psi_t(y) &:= \min_{x \in X} \bF_t(x,y) \left[= \bF_t(\wh x_t(y), y) \right], \quad
\wh x_t(y) := \argmin_{x \in X} \bF_t(x,y),
\end{ald}
\end{eq}
with 
Clearly,~$\psi_t(y)$ is~$\lamy$-strongly concave with~$\lamy = \bey/\Ry$.
On the other hand, by Danskin's theorem (see,~e.g.,~\cite[Lem.~24]{nouiehed2019solving}),~$\psi_t(y)$ is continuously differentiable with 
\begin{eq}
\label{eq:psi-grad}
\nabla \psi_t(y) = \partial_y \bF(\wh x_t(y),y) = \partial_y F(\wh x_t(y),y) - \lamy(y-\by),
\end{eq}
and~$\nabla \psi_t(y)$ is~$(\Lyy^+ + \lamy)$-Lipschitz with~$\Lyy^+$ defined in~\cref{def:lyy-plus}. 

\proofstep{2}. We now focus on the properties of the point~$\wt x_t(y)$ returned when calling
$
\AppGrad(y,x_{t-1},\by,\gamx,\lamy,\To,\So),
$
cf.~line~\ref{line:call-approx-grad} of~\cref{alg:fsp}, as well as the corresponding pair~$[\wt \psi_t(y), \wt\nabla \psi_t(y)]$, cf.~\cref{eq:psi-inexact-oracle}. 
Note that the function value~$\wt \psi_t(y)$ is never computed in~\cref{alg:fsp} and we only use it in the analysis.
Inspecting the pseudocode of~\AppGrad (\cref{alg:grad-approx}), we see that~$\wt x_t(y)$ corresponds to the approximate minimizer of~$\bF_t(x,y)$ (thus also~$F_t(x,y)$) in~$x$, obtained by running restarted FGM  (\cref{alg:fgm-restart}) starting from~$x_{t-1}$, with stepsize~$\gam = 1/(3\Lxx)$,~$\To = 11$ inner loop iterations, and the number of restarts~$\So$ given in~\cref{eq:inner-loop-params}. 
Observe that minimizing~$F_t(\cdot,y)$ corresponds to computing the proximal operator~$\x^+_{\gam F(\cdot,y),X}(x_{t-1})$ for the function~$F(\cdot,y)$ which is~$\Lxx$-smooth. 
Hence, due to our choice of input parameters, the premise of~\cref{prop:ppm-via-fgm} is satisfied; applying it with our choice of~$\So$ yields
\begin{align}
\label{eq:approximation-accuracy-grads}
\Dx(\wt x_t(y), \gx F_t (\wt x_t(y),y), \Lxx) &\le \frac{\ex}{2},\\
\label{eq:approximation-accuracy-points}
\| \wt x_t(y) - \wh x_t (y) \| &\le \min\left[ \frac{\ex}{6\Lxx}, \frac{\delta}{8\Lxy\Ry} \right],\\
\label{eq:approximation-accuracy-values}
\bF_t (\wt x_t(y),y)  - \bF_t (\wh x_t(y),y) &\le \min\left[\frac{\ex^2}{24\Lxx},  \frac{\delta}{2} \right].
\end{align}
Here,~\cref{eq:approximation-accuracy-grads} and the first respective terms in~\cref{eq:approximation-accuracy-points}--\eqref{eq:approximation-accuracy-values} are due to the first of three terms in brackets under logarithm in~\cref{eq:inner-loop-params}, cf.~\cref{eq:ppm-via-fgm-params}, combined with a very crude uniform over~$y \in Y$ estimate
\begin{eq}
\label{eq:recursive-bound}
\bF_t(x_{t-1},y) - \min_{x \in X} \bF_t(x,y) \le 3\Gap + 2\Theta + 6\bey \Ry.
\end{eq}
(We defer the proof of~\cref{eq:recursive-bound} to appendix.)
On the other hand, the second respective estimates in~\eqref{eq:approximation-accuracy-points}--\eqref{eq:approximation-accuracy-values}  correspond to the two remaining terms in brackets under logarithm in~\cref{eq:inner-loop-params}, cf.~\cref{eq:ppm-via-fgm-params}, combined with the following easy-to-verify relations:
\[
\frac{\Lxx}{\ex^2} = \frac{1}{12\delta} \quad \Longleftrightarrow \quad \frac{\ex^2}{24\Lxx} = \frac{\delta}{2},
\]
\[
\left\{ 
\frac{2\Theta^+}{\delta^2} \ge \frac{16 \Lxy^2 \Ry^2}{9 \Lxx \delta^2} = \Lxx \left(\frac{8\Lxy \Ry}{6\Lxx\delta} \right)^2 =: \frac{\Lxx}{(\ex')^2}
\right\} 
\quad \Longrightarrow \quad
\frac{\ex'}{6\Lxx} = \frac{\delta}{8\Lxy \Ry}.
\]
Now,~\eqref{eq:approximation-accuracy-grads}--\eqref{eq:approximation-accuracy-values} have two consequences.
First, by~\cref{eq:approximation-accuracy-values} we immediately have
\begin{eq}
\label{eq:ppm-via-fgm-max-min-step-progress}
\bF(\wt x_t(y), y) + \Lxx \| \wt x_t(y)  - x_{t-1} \|^2 - \frac{\ex^2}{24\Lxx}  \le \bF(x_{t-1},y),
\end{eq}
which mimics~\cref{eq:ppm-primal-dual-simple-half}. 
The bound~\cref{eq:ppm-via-fgm-max-min-step-progress} will be our departure point when bounding~$\Dx$ later on.
Second, the second respective terms in the right-hand side of~\cref{eq:approximation-accuracy-points}--\eqref{eq:approximation-accuracy-values} together ensure that  the pair~$[-\wt\psi_{t}^\delta(y), -\wt \nabla \psi_t(y)]$ with
\begin{eq}
\label{eq:inexact-oracle-psi-xi}
\begin{ald}
\wt\psi_{t}^\delta(y) := \bF_t(\wt x_t(y),y) + {\delta}/{4}, \quad 
\wt\nabla\psi_t(y) =  \partial_y \bF(\wt x_t(y),y)
\end{ald}
\end{eq}
is a~$\delta$-inexact first-order oracle for~$-\psi_t(y)$ in the sense of Definition~\ref{def:inexact-oracle}, namely,
\begin{eq}
\label{eq:inexact-oracle-for-psi}
0 \le -\psi_t(y') + \wt\psi_{t}^\delta(y) + \langle \wt\nabla \psi_t(y), y' - y \rangle  \le \tfrac{1}{2} (\Lyy^+ + \lamy) \|y'-y\|^2 + \delta, \quad \forall y,y' \in Y,
\end{eq}
where we used that~$\psi_t$ is~$(\Lyy^+ + \lamy)$-smooth (the proof of~\cref{eq:inexact-oracle-for-psi} is deferred to appendix).

\proofstep{3}.
Now consider the actual update performed in the for-loop of~\cref{alg:fsp}:
\begin{eq}
\label{eq:actual-updates-proof}
\begin{ald}
y_t \approx \arg\max_{y \in Y} \psi_t(y), \quad 
x_t = \wt x_t(y_t),
\end{ald}
\vspace{-0.1cm}
\end{eq}
where the precise meaning of~``$\approx$'', is
$
y_{t} = \FGMR (\by, Y, \gamy, \bTy, \Sy, -\wt\nabla\psi_t(\cdot)),
$
cf.~line~\ref{line:dual-update}.
In other words,~$y_t$ is obtained by running Algorithm~\ref{alg:fgm-restart} with~$\delta$-inexact gradient~$\wt \nabla \psi_t(\cdot)$, starting from~$\by \in Y$, with~$\bTy$ iterations in the inner calls of FGM and~$\Sy$ restarts,~$\bTy$ and~$\Sy$ being given in~\cref{eq:outer-loop-params}.
Recall that~$\psi_t(\cdot)$ is~$(\Lyy^+ + \lamy)$-smooth and~$\lamy$-strongly convex with~$\lamy = \bey/\Ry$, and 
\begin{eq}
\label{eq:delta-required}
\delta \stackrel{\cref{eq:delta-value}}{\le} \frac{\Theta}{2\bTy^3} \le \frac{(\Lyy^+ + \lamy) \Ry^2}{2\bTy^{3}}, 
\end{eq}
i.e., the condition in~\cref{eq:fgm-oracle-accuracy-bound} is satisfied. 
By our choice of~$\bTy$ and~$\Sy$ in~\cref{eq:outer-loop-params}, and due to~\cref{cor:fgm-restart-bound}, we get
\vspace{-0.2cm}
\begin{eq}
\label{eq:dual-update-accuracy}
\begin{ald}
\|y_t - y_t^* \| \le \frac{\bey}{3\Lyy^+}, \quad 
\psi_t(y_t^*)-\psi_t(y_t) \le\min\left[\frac{\bey^2}{18\Lyy^+\vphantom{\bTy^2}},\frac{\ex^2}{18\Lxx \bTy^2}\right],\\ 
\end{ald}
\end{eq}
and~$\Dy(y_t, -\nabla \psi_t(y_t), \Lyy^+ + \lamy) \le {\bey}/{3},$
where~$y_t^*$ is the exact maximizer of~$\psi_t$ (cf.~\cref{eq:fgm-approximation-bounds}). 
Here we used the first lower bound in~\cref{eq:outer-loop-params} for~$\Sy$ to obtain all estimates except for the second estimate of~$\psi_t(y_t^*) - \psi_t(y_t)$, and for this latter estimate we used the second bound in~\cref{eq:outer-loop-params} for~$\Sy$ and the last bound in~\cref{eq:delta-value} for~$\delta$, and did a series of estimates: 
\vspace{-0.3cm}
\[
\begin{ald}
\Sy \stackrel{\cref{eq:outer-loop-params}}{\ge} \log_2\left(  \frac{(\Theta^+)^2}{\delta^2} \right) 
\stackrel{\cref{eq:delta-value}}{\ge} \log_2 \left( \frac{{\bTx}^{\vphantom 2} {\bTy}^2 \Theta^+} {\Gap} \right) 
&\ge \log_2 \left( \frac{10 \Lxx {\bTy}^2 \Lyy^+ \Ry^2} {\ex^2} \right) \\
&\ge 2\log_2 \left( \frac{3 \Lyy^+ \Ry}{\ey'} \right), \;
(\ey')^2 := \frac{\ex^2  \Lyy^+}{\Lxx \bTy^2}.
\end{ald}
\]
Now, by the proximal PL-lemma~(\cite[Lem.~1]{karimi2016linear}),~$\Dy(y, g, L)$ 
is non-decreasing in~$L$, so
\vspace{-0.1cm}
\begin{eq}
\label{eq:psi-stationarity}
\Dy(y_t, -\nabla \psi_t(y_t), \Lyy) \le {\bey}/{3}.
\end{eq}
Due to~\cref{eq:psi-grad} and the Lipschitzness of~$\gy F(\cdot, y)$ and~$\prox_{y_t,Y}(\cdot)$,~$x_t = \wt x_t(y_{t})$ satisfies
\vspace{-0.1cm}
\begin{equation*}
\begin{ald}
&\Dy^2(y_t,-\gy F(x_t,y_t),\Lyy) \\
&\stackrel{(a)}{\le} \Dy^2(y_t,-\gy F(x_t,y_t),2\Lyy) 
=4\Lyy \max_{y \in Y} \big[\big\langle \gy F(x_t, y_t), y - y_t  \big\rangle  - {\Lyy}\|y - y_t\|^2 \big] \\
&\le 4\Lyy \max_{y \in Y} \big[\big\langle \gy F(x_t, y_t) - \py F(\wh x_t(y_t), y_t) + \tfrac{\ey}{\Ry}(y_t - \by), y - y_t \big\rangle  - \tfrac{\Lyy}{2}\|y - y_t\|^2 \big] \\
&\quad\quad\quad+ 4\Lyy \max_{y \in Y} \big[\big\langle \py F(\wh x_t(y_t), y_t) - \tfrac{\ey}{\Ry}(y_t - \by), y - y_t \big\rangle  - \tfrac{\Lyy}{2}\|y - y_t\|^2 \big] \\
&= 4\Lyy \max_{y \in Y} \big[\big\langle \gy F(x_t, y_t) - \py F(\wh x_t(y_t), y_t) + \tfrac{\ey}{\Ry}(y_t - \by), y - y_t \big\rangle  - \tfrac{\Lyy}{2}\|y - y_t\|^2 \big] \\
&\quad\quad\quad + 2\Dy^2(y_t,-\nabla \psi_t(y_t), \Lyy) \\
&\stackrel{(b)}{\le} 4\Lyy \max_{y \in Y} \big[\big\langle \gy F(x_t, y_t) - \py F(\wh x_t(y_t), y_t) + \tfrac{\ey}{\Ry}(y_t - \by), y - y_t \big\rangle  - \tfrac{\Lyy}{2}\|y - y_t\|^2 \big] + \tfrac{2\ey^2}{9} \\ 
&\stackrel{(c)}{\le} 2\big\|\gy F(x_t, y_t) - \py F(\wh x_t(y_t), y_t) + \tfrac{\ey}{\Ry}(y_t - \by)\big\|^2 + \tfrac{2\ey^2}{9} \\
&\stackrel{(d)}{\le} 4\left\|\gy F(x_t, y_t) - \py F(\wh x_t(y_t), y_t) \right\|^2 + \tfrac{4\ey^2}{\Ry^2}\left\|y_t - \by\right\|^2 + \tfrac{2\ey^2}{9} \\
&\stackrel{(e)}{=} 4 \Lxy^2 \left\|\wt x_t(y_t) - \wh x_t(y_t)\right\|^2 + 16\ey^2 + \tfrac{2}{9} \ey^2
\stackrel{(f)}{\le} \tfrac{\delta^2}{16\Ry^2} + 16\ey^2 + \tfrac{2}{9}\ey^2
\stackrel{(g)}{\le} 21 \ey^2.
\end{ald}
\end{equation*}
\normalsize
Here 
in~$(a)$ we used that~$\Dy(y,g,L)$ is non-decreasing in~$L$~(\cite[Lem.~1]{karimi2016linear}); 
in~$(b)$ we used~\cref{eq:psi-stationarity}; 
in~$(c)$ we used Young's inequality;
in~$(d)$ we used the Cauchy-Schwarz inequality; 
in~$(e)$ we used the Lipschitzness of~$F$; 
in~$(f)$ we used~\cref{eq:approximation-accuracy-points};
in~$(g)$ we used our choice of~$\delta$ in~\cref{eq:delta-value}.
Thus,~$(x_t, y_t)$ is kept~$5\ey$-stationary in~$y$ at any iteration~$t$.

\proofstep{4}.
We now revisit~\cref{eq:ppm-via-fgm-max-min-step-progress}. 
Applying it to~$y = y_t$, we get
\begin{eq}
\label{eq:ppm-primal-dual-simple-half-proof}
\bF(x_t, y_t) + \Lxx \|x_t - x_{t-1}\|^2 - \frac{\ex^2}{24\Lxx} \le \bF(x_{t-1},y_t),
\end{eq}
which mimics~\cref{eq:ppm-primal-dual-simple-half}.
Our goal, however, is to mimic~\cref{eq:ppm-primal-dual-simple}, for which we must lower-bound, up to a small error,~$\bF(x_t, y_t)$ via~$\bF(x_t, y_{t+1})$, or, equivalently,~$\bF_t(x_t, y_t)$ via~$\bF_t(x_t, y_{t+1})$.
First, 
\begin{eq}
\label{minmax-argument-prev}
\bF_t(x_t, y_{t+1})  \le \max_{y \in Y} \bF_t(x_t, y) = \vphi_t(x_t),
\end{eq}
where~$\vphi_t(x) := \textstyle\max_{y \in Y} \bF_t(x,y)$ is the primal function in the saddle-point problem~\cref{eq:minmax-conceptual-updates}.
On the other hand, denoting~$x_t^* = \wh x_t(y_t^*)$, so that~$(x_t^*, y_t^*)$ is the unique saddle point in~\cref{eq:minmax-conceptual-updates}, we have
\begin{align}
\bF_t(x_t, y_t) 
\equiv \bF_t(\wt x_t(y_t), y_t) 
\ge \bF_t(\wh x_t(y_t), y_t) 
= \psi_t(y_t)  
&\stackrel{\eqref{eq:dual-update-accuracy}}{\ge} \psi_t(y_t^*) - \frac{\ex^2}{18\Lxx \bTy^2} \nn
&\stackrel{\hphantom{\eqref{eq:dual-update-accuracy}}}{\ge} \vphi_t(x_t^*) - \frac{\ex^2}{18\Lxx \bTy^2}.
\label{eq:minmax-argument-this}
\end{align}
It remains to compare~$\vphi_t(x_t)$ and~$\vphi_t(x_t^*)$. 
Combining~$\bF_t(x_t^*,y_t^*) \ge \bF_t(x_t^*,y_t)$ with the previous inequality,
and observing that~$\bF_t(\cdot, y_t)$ is~$\Lxx$-strongly convex and minimized at~$\wh x_t(y_t)$, we obtain
$
\|\wh x_t(y_t) - x_t^*\|^2  \le \frac{\ex^2}{9\Lxx^2 \bTy^2}.
$
On the other hand,~\cref{eq:approximation-accuracy-points} applied to~$y = y_t$ gives
\begin{eq}
\label{eq:approximation-of-hat-x}
\left\| x_t - \wh x_t(y_t) \right\|^2 \le \left(\frac{\delta}{8\Lxy\Ry}\right)^2 \le \frac{\Gap (\Theta^+ - \Theta)}{64 \Lxy^2 \Ry^2 \bTy^2 \bTx}
= \frac{\ex^2}{640\Lxx^2\bTy^2},
\end{eq}
where we used the last expression in~\cref{eq:delta-value} for~$\delta$.
Combining these results, we get
\[
\|x_t - x_t^*\|^2 \le \frac{0.12 \ex^2}{\Lxx^2 \bTy^2}.
\]
Now,~$\vphi_t$ is~$\left(3\Lxx + {\Lxy^2}/{\lamy^{\vphantom{2}}}\right)$-smooth by Danskin's theorem, and minimized at~$x_t^*$. 
Thus
\begin{eq}
\label{eq:iteration-acc-primal-gap} 
\vphi_t(x_t) - \vphi_t(x_t^*) \le \frac{3}{2} \left(\Lxx + \frac{\Lxy^2}{\lamy^{\vphantom{2}}}\right) \|x_t - x_t^*\|^2 
\le \frac{0.18\ex^2}{\Lxx} \left(1+ \frac{\Lyy^+}{\lamy \bTy^2} \right)
\le \frac{0.19\ex^2}{\Lxx},
\end{eq}
where in the last step we plugged in~$\bTy$ from~\cref{eq:outer-loop-params}.
Returning to~\cref{minmax-argument-prev}--\cref{eq:minmax-argument-this}, we get
$
\bF(x_t, y_t) \ge \bF(x_t, y_{t+1}) - {0.25\ex^2}/{\Lxx}.
$
Combining this with~\cref{eq:ppm-primal-dual-simple-half-proof} we finally get the desired analogue of~\cref{eq:ppm-primal-dual-simple}:
\begin{eq}
\label{eq:ppm-primal-dual-simple-proof}
\bF(x_t, y_{t+1}) + \Lxx \|x_t - x_{t-1}\|^2 - \frac{7\ex^2}{24\Lxx} \le \bF(x_{t-1},y_t).
\end{eq}
This bound can be iterated, and we can proceed as in~\cref{sec:ppm-via-fgm}. 
First we mimic~\cref{eq:prox-point-argument-approx}:
\begin{align}
\min_{t \in [\bTx]} \|x_t - x_{t-1}\|^2
\le \frac{1}{\bTx} \sum_{t \in [\bTx]} \| x_t - x_{t-1} \|^2 
&\le \frac{\bF(x_0, y_1) - \bF(x_{\bTx}, y_{\bTx})}{\Lxx\bTx} + \frac{7\ex^2}{24 \Lxx^2} \notag \\
&\le \frac{\Gap + 2\ey \Ry}{\Lxx \Tx} + \frac{7\ex^2}{24 \Lxx^2} 
\le \frac{5\ex^2}{12\Lxx^2},
\label{eq:prox-point-argument-approx-proof}
\end{align}
where we used the estimates~\cref{eq:gap-estimate} and plugged in~$\Tx$. 
It remains to mimic~\eqref{eq:stationarity-by-closeness-approx}:
\[
\begin{ald}
\Dx^2(x_t, \gx F(x_t, y_t), \Lxx) 
&\le \Dx^2(x_t, \gx F(x_t, y_t), 2\Lxx) \\
&\equiv 4\Lxx \max_{x' \in X} \left[- \lang \gx F(x_t, y_t), x' - x_t \rang - {\Lxx}\|x' - x_t\|^2 \right] \\
&\le 4\Lxx \max_{x' \in X} \left[- \lang \gx F_t (x_t, y_t), x' - x_t \rang - \tfrac{1}{2} \Lxx \|x' - x_t\|^2 \right] \\
&\quad + 4\Lxx \max_{x' \in X} \left[ \lang 2 \Lxx (x_t - x_{t-1}), x' - x_t \rang - \tfrac{1}{2} \Lxx \|x' - x_t\|^2 \right] \\
&= 2 \Dx(x_t, \gx F_t (x_t,y_t), \Lxx) \\ 
&\quad + 4\Lxx \max_{x' \in X} \left[ \lang 2 \Lxx (x_t - x_{t-1}), x' - x_t \rang - \tfrac{1}{2} \Lxx \|x' - x_t\|^2 \right] \\ 
&\stackrel{(a)}{\le} {\ex^2}/{2} + 4\Lxx^2 \max_{x' \in \XX} \left[ 2 \lang x_t - x_{t-1}, x' - x_t \rang - \tfrac{1}{2} \|x' - x_t\|^2 \right] \\ 
&\stackrel{(b)}{\le} {\ex^2}/{2} + 8\Lxx^2 \| x_t - x_{t-1} \|^2 
\stackrel{(c)}{\le} \left( {1}/{2} + {10}/{3} \right) \ex^2 \le 4\ex^2,
\end {ald}
\]
where in~$(a)$ we used~\cref{eq:approximation-accuracy-grads} with~$y = y_t$, in~$(b)$ we used Young's inequality, and~$(c)$ was due to~\cref{eq:prox-point-argument-approx-proof}.
Combining this with the result of~\proofstep{3}, we conclude that~$(x_\tau, y_\tau)$ with~$\tau \in \argmin_{t \in \bTx} \|x_t - x_{t-1}\|^2$ is~$(2\ex,5\ey)$-FNE.
Moreover, we have performed~$\left\lceil \To \So  \Sy \bTx \bTy \right\rceil$ iterations of~FGM (in the for-loop of~\cref{alg:fgm}) in total, with one computation of~$\nabla F$ and at most two projections on~$Y$ and $X$ at each iteration. \hfill\proofbox

\odima{
\section{Guarantees for the Moreau envelope}
\label{sec:moreau}
We now consider the standard {\em Moreau envelope} (see~\cite{thekumparampil2019efficient,lin2019gradient}) of the primal function~$\vphi(x) = \max_{y \in Y} F(x,y)$, cf.~\cref{def:prim-function}:
\begin{equation}
\label{def:moreau-envelope}
\vphi_{2\Lxx}(x) := \min_{x' \in X} \left[ \vphi(x') + \Lxx \|x' - x\|^2 \right].
\end{equation}
Clearly,~$\vphi$ is~$\Lxx$-weakly convex, thus the minimized function in~\cref{def:moreau-envelope} is~$\Lxx$-strongly convex.
Focusing on the Moreau envelope makes sense in the applications where one is only interested in the ``primal'' accuracy of solving~\cref{opt:min-max}. 
A common practice, in the~$x$-unconstrained case ($X = \XX$), is then to use the primal component~$\wh x$ of an approximate Nash equilibrium~$(\wh x, \wh y)$ as a candidate near-stationary point, measuring the accuracy by~$\|\nabla \vphi_{2\Lxx}(\wh x)\|$. 
When~$X = \XX$, passing to the Moreau envelope can be motivated as follows. 
On the one hand,~$\vphi_{2\Lxx}$ has Lipschitz gradient on~$X$ (by Danskin's theorem), and we can ``ignore'' the non-differentiability of~$\vphi$. 
On the other hand,~$\|\nabla \vphi_{2\Lxx}(\wh x) \| \le \ex$ implies that the point~$x^+ = x^+(\wh x)$ delivering the minimum in~\cref{def:moreau-envelope} for~$x = \wh x$ (formally, $x^+_{\lambda\vphi, \XX}(\wh x)$ with~$\lambda = \tfrac{1}{2\Lxx}$, cf.~\cref{def:prox-map}) satisfies
$\|x^+ - \wh x\| = O(\tfrac{\ex}{\Lxx})$ and~$\min_{\xi \in \partial \vphi(x^+)} \|\xi\| \le \ex$,
see,~e.g.,~\cite{rockafellar2015convex}.
In other words, any~$\ex$-stationary point for the Moreau envelope is within~$O(\ex/\Lxx)$ distance from a point at which~$\vphi$ has an~$\ex$-small subgradient.
We now extend this result to the case $X \subseteq \XX$. 
\begin{proposition}
\label{prop:moreau-to-primal}
Let~$\phi: X \to \R$ be~$L$-weakly convex, and define its standard Moreau envelope~$\phi_{2L}(x) = \min_{x' \in X}[\phi(x') + L \|x'-x\|^2]$, cf.~\cref{def:moreau-envelope}. Then:\\
1. We have~$\|\nabla\phi_{2L}(x)\| = \Dx(x,\nabla\phi_{2L}(x),2L) = \Ex(x,\nabla\phi_{2L}(x),2L)$ for any~$x \in X$.\\
2. We have~$\nabla \phi_{2\Lxx}(\wh x) = 2\Lxx (\wh x - x^+)$, where~$x^+ = \argmin_{x' \in X} \left[ \phi(x') + \Lxx \|x' - \wh x\|^2 \right]$.\\
\quad Thus~$\|\nabla\phi_{2\Lxx}(\wh x)\| \le \ex$ implies~$\|x^+ - \wh x\| \le {\ex}/{(2\Lxx)}$ and~$\displaystyle\min_{\xi \in \partial \phi(x^+)} \Dx(x^+,\xi,2\Lxx) \le \ex.$
\end{proposition}
This result is proved in appendix.
\cref{prop:moreau-to-primal} motivates the task of finding a point~$\wh x$ with a small norm of the Moreau envelope~$\|\nabla \vphi_{2\Lxx}(\wh x)\|$. 
The recent work~\cite{thekumparampil2019efficient} proposes an algorithm that directly produces such a point in~$O(\ex^3)$ first-order oracle calls in the primally-unconstrained setup ($X = \XX$). 
The work~\cite{lin2020near} uses a different approach. They first find an~$O(\ex,\ey)$-FNE for~\eqref{opt:min-max} with respect to the {\em weak} criterion, i.e.,~$(\wh x, \wh y)$ such that~\cref{eq:stationarity} holds with~$\Dx, \Dy$ replaced with~$\Ex, \Ey$ respectively. 
Then they use the result~\cite[Prop.~4.12]{lin2019gradient} that claims to guarantee (again in the case~$X = \XX$) that~$\|\nabla \vphi_{2\Lxx}(\wh x)\| = O(\ex)$ whenever~$\ey = O(\ex^2)$. 
Let us rephrase the claim in~\cite{lin2019gradient}.
\begin{proposition}[{\cite[Prop.~4.12]{lin2019gradient}}]
\label{prop:fne-to-moreau-wrong}
Assuming~\cref{ass:smoothness} and~$X = \XX$, one has that
\begin{eq}
\label{eq:fne-to-moreau-wrong}
\|\nabla \vphi_{2\Lxx}(\wh x)\|^2 = O({\|\gx F(\wh x, \wh y)\|^2} + \Lxx \Ry \hEy),
\end{eq}
where~$\hEy := \Ey(\wh y, -\gy F(\wh x, \wh y), \Lyy)$.
In particular,~$\|\nabla \vphi_{2\Lxx}(\wh x)\| = O(\ex)$ provided that~$\|\gx F(\wh x, \wh y)\| \le \ex$ and~$\hEy \le \ex^2/(\Lxx\Ry)$.
\end{proposition}
Inspecting the results in~\cite{lin2020near}, we conclude that their algorithm outputs an~$(\ex,\ey)$-FNE (with respect to the weak measure~$\Ey$ in $y$) in~$\wt O(\Tx \Ty)$ oracle calls. 
By~\cref{prop:fne-to-moreau-wrong}, this translates to finding an~$\ex$-stationary point for the Moreau envelope in
\begin{eq}
\label{eq:moreau-complexity}
\wt O\left(\frac{\Gap{\Lxx^{\vphantom+}}^{3/2}{\Lyy^+}^{1/2}\Ry^{\vphantom+}}{\ex^3}\right)
\end{eq}
oracle calls. However, our careful inspection of the proof of~\cite[Prop.~4.12]{lin2019gradient} only allowed to verify~\cref{eq:fne-to-moreau-wrong} in the unconstrained case~$Y = \YY$ (and replacing~$\Ry$ with the distance~$\| \wh y - y^o\|$ for some~$y^o \in \Argmax_{y \in Y} F(\wh x, y)$), so that~$\hEy$ becomes~$\|\gy F(\wh x, \wh y)\|$. 
The underlying issue is that the proof relies on the bound
\begin{eq}
\label{eq:constrained-concavity-wrong}
F(\wh x, y^o) - F(\wh x, \wh y) \le \langle \wh \zeta_{\y}, y^o - \wh y \rangle,
\end{eq}
where~$\wh \zeta_{\y} := \Lyy(\Pi_{Y}[\wh y + \tfrac{1}{\Lyy} \gy F(\wh x, \wh y)] - \wh y)$ is the negative proximal gradient of~$-F(\wh x, \cdot)$ at~$\wh y$;
this gives the term~$\Ry  \Ey(\wh y, -\gy F(\wh x, \wh y), \Lyy)$ in~\cref{eq:fne-to-moreau-wrong} by Cauchy-Schwarz.
However,~\cref{eq:constrained-concavity-wrong} can be invalid when~$Y \ne \YY$.
The following result allows to rectify this.
\begin{lemma}
\label{lem:constrained-concavity} 
Let $h: Y \to \R$ be differentiable and concave, and define~$\zeta^L(y) := L(\Pi_{Y}[y + \tfrac{1}{L} \nabla h(y)] - y)$ for~$L > 0$.
Then for any~$y,y' \in Y$ and~$L > 0$, one has that
\begin{eq}
\label{eq:constrained-concavity} 
h(y') - h(y) \le \langle \zeta^L(y), y' - y \rangle + \tfrac{1}{2L} \left[ \Dy^2(y, -\nabla h(y),L) - \Ey^2(y, -\nabla h(y),L) \right].
\end{eq}
\end{lemma}

\begin{rem}
\label{rem:moreau-tightness}
When~$Y = \YY$, the second term in the right-hand side of~\cref{eq:constrained-concavity} vanishes, and we recover~\cref{eq:constrained-concavity-wrong} by putting~$h(y) = F(\wh x,y)$.
Meanwhile, in the constrained case~\cref{eq:constrained-concavity} is tight up to a constant factor; moreover,~\cref{eq:constrained-concavity-wrong} can be violated with {\em arbitrary gap} due to 
the additional term in the right-hand side, which can be arbitrarily large (while the inner product term remains fixed).
Indeed, consider the following problem for~$a \ge 0$:
\[
\max_{y \in [-1,0]} [h(y) := -\tfrac{1}{2}{(y-a)^2}].
\]
Clearly,~$y^o = 0$ is the unique maximizer of~$h(\cdot)$ on~$Y = [-1,0]$.
On the other hand, by simple algebra we verify that, with~$L = 1$ (which corresponds to the smoothness of~$h$), any~$a \ge 0$ and~$\veps \in [0,1]$, the point~$\wh y = -\veps$ satisfies~$h(y^o) - h(\wh y) = \tfrac{1}{2}\veps^2 + a\veps$,~$\zeta^L(\wh y) = \veps$,
\[
\langle \zeta^L(\wh y), y^o - \wh y \rangle = \veps^2, \quad
\Ey^2(\wh y, -\nabla h(\wh y),L) = \veps^2, \quad 
\Dy^2(\wh y, -\nabla h(\wh y),L) = 2a\veps + \veps^2.
\] 
Thus, for this instance~\cref{eq:constrained-concavity} with~$y' = y^o$ and~$y = \wh y$ is almost attained: the left-hand side is equal to~$\tfrac{1}{2}\veps^2 + a \veps$ and the right-hand side to~$\veps^2 + a \veps$. 
Moreover, the term~$\tfrac{1}{2L}[\Dy^2(\wh y, -\nabla h(\wh y),L) - \Ey^2(\wh y, -\nabla h(\wh y),L)] = a\veps$ can be made arbitrarily large (by increasing~$a$) without changing the term~$\langle \zeta^L( \wh y), y^o - \wh y \rangle = \veps^2$.
\end{rem}
As we noted before, the error in~\cref{eq:constrained-concavity} seems to invalidate~\cref{prop:fne-to-moreau-wrong}, and thus the complexity estimate~\cref{eq:moreau-complexity}.
Indeed,~\cref{eq:fne-to-moreau-wrong} in fact gains the additional term under~$O(\cdot)$ 
in the right-hand side, 
and this term can be arbitrarily large\footnote{Our attempts to obtain an alternative proof of~\cref{prop:fne-to-moreau-wrong} without using~\cref{eq:constrained-concavity} have failed.} when~$\hEy \le \ey$.
Fortunately, the complexity estimate~\cref{eq:moreau-complexity} can be obtained in the fully constrained setup, by using that~\cref{alg:fsp} produces an~$(\ex,\ey)$-FNE in the {\em strong sense} (cf.~\cref{def:stationarity}). 
\begin{proposition}
\label{prop:fne-to-moreau}
Assume~\cref{ass:smoothness} and let~$\hEy$ be as defined in~\cref{prop:fne-to-moreau-wrong}, then
\begin{eq}
\label{eq:fne-to-moreau}
{\|\nabla \vphi_{2\Lxx}(\wh x)\|^2} = O\left({\Dx^2(\wh x, \gx F(\wh x, \wh y), \Lxx)} + \Lxx\Ry \hEy + {\Lxx(\hDy{}^2 - \hEy{}^2)}/{\Lyy}\right), 
\end{eq}
where~$\hDy := \Dy(\wh y, -\gy F(\wh x, \wh y), \Lyy)$.
As a result,~$\|\nabla \vphi_{2\Lxx}(\wh x)\| = O(\ex)$ whenever~$(\wh x, \wh y)$ is~$(\ex,\ey)$-FNE for~\cref{opt:min-max}, in the sense of~\cref{eq:stationarity}, with
\begin{eq}
\label{eq:fne-to-moreau-concave-ey}
\ey
\le \min\bigg[ \frac{\ex^2}{\Lxx\Ry}, \ex\sqrt{\frac{\Lyy}{\Lxx}}\,\bigg].
\end{eq}
\end{proposition}
Recalling~\cref{th:upper-bound} (cf.~\cref{eq:complexity-simple}), we conclude that~\cref{alg:fsp}, when run with 
\begin{equation}
\label{eq:fne-to-moreau-concave-ey-final}
\ey
= \frac{\ex^2}{\Lxx\Ry},
\end{equation}
produces~$\wh x \in X$ that satisfies~$\|\nabla \vphi_{2\Lxx}(\wh x)\| = O(\ex)$ in
\begin{eq}
\label{eq:moreau-complexity-final}
\boxed{
\wt O\left(\frac{\Gap{\Lxx^{\vphantom+}}^{3/2}{\Lyy^+}^{1/2}\Ry^{\vphantom+}}{\ex^3} \right)
\; \text{oracle calls}.
}
\end{eq}
Indeed, under~\cref{eq:fne-to-moreau-concave-ey-final} there are two possibilities. 
If~$\ex^2/\Lxx = O(\Lyy{\vphantom 2} \Ry^2)$, then we have~\cref{eq:fne-to-moreau-concave-ey} and can apply~\cref{prop:fne-to-moreau}. 
On the other hand, in the case~$\Lyy^{\vphantom 2}\Ry^2 = O(\ex^2/\Lxx)$ we can directly use~\cref{eq:fne-to-moreau}
combined with the bound
$\Dy(y_t, -\gy F(x_t, y_t), \Lyy)  = O(\Lyy \Ry)$;
this bound follows from~\cref{th:devolder} (note that~$\bTy > 1$, cf.~\cref{eq:outer-loop-params}) combined with~\cref{eq:-fgm-near-stationarity-argument}.

}
\odima{
\section*{Acknowledgments}
We thank Babak Barazandeh for technical discussions, for discovering the issue with~\cite[Prop.~4.12]{lin2019gradient} (cf.~\cref{eq:constrained-concavity-wrong}) and for sketching the proof of~\cref{eq:constrained-concavity}. 
}

\footnotesize
\bibliographystyle{siamplain}
\bibliography{references}

\normalsize

\vspace{0.2cm}

\appendix
\section{Deferred proofs}
\label{sec:deferred-proofs}

\subsection{Verification of~\cref{eq:inexact-oracle-for-psi}}
By concavity and~$(\Lyy^+ + \lamy)$-smoothness of~$\psi_t$, 
\[
0 \le -\psi_t(y') + \psi_t(y) + \lang \nabla \psi_t(y), y' - y \rang  \le \tfrac{1}{2}(\Lyy^+ + \lamy)\|y'-y\|^2, \quad \forall y,y' \in Y.
\]
By~\cref{eq:conceptual-updates-proof-dual,eq:approximation-accuracy-values,eq:inexact-oracle-psi-xi},
${\delta}/{4} \le \wt \psi_t^\delta(y) - \psi_t(y) \le {3\delta}/{4}$ for all~$y \in Y.$
On the other hand, by the second part of~\cref{eq:approximation-accuracy-points},
\[
\| \wt\nabla \psi_t(y) - \nabla \psi_t(y) \| = \left\| \partial_y F(\wt x_t(y),y) - \partial_y F(\wh x_t(y),y) \right\| \le \Lxy \| \wh x_t(y) - \wt x_t(y) \| \le \frac{\delta}{8\Ry},
\]
hence, as~$\|y'-y\| \le 2\Ry$ for~$y',y \in Y$, we get 
$
-{\delta}/{4} \le \langle \wt\nabla \psi_t(y) - \nabla \psi_t(y), y' - y \rangle \le {\delta}/{4}.
$
We obtain~\eqref{eq:inexact-oracle-for-psi} by summing up the two-sided inequalities above. \hfill\proofbox

\subsection{Verification of~\cref{eq:recursive-bound}}

Let~$\vphi_t(x) = \textstyle\max_{y \in Y} \bF_t(x,y)$ be the primal function of the saddle-point problem in~\cref{eq:minmax-conceptual-updates}.
Then
\[
\vphi_t(x) - 2\Lyy^{\vphantom 2} \Ry^2 \le F_t(x,y) \le \vphi_t(x)
\]
by bounding the variation of a smooth function~$F(x,\cdot)$ over~$y \in Y$, whence
\begin{align}
F_t(x_{t-1},y) - \min_{x \in X} F(x,y) 
&\le \vphi_t(x_{t-1}) - \min_{x} \vphi_t(x) + {2\Lyy^{\vphantom 2} \Ry^2} \nn
\label{eq:recursive-proof-first}
&\le \vphi_t(x_{t-1}) - \min_{x \in X} \vphi(x) + {2\Lyy^{\vphantom 2} \Ry^2} + 2\bey \Ry,
\end{align}
where we used that~$\bF_t(x,y) \ge F(x,y) - 2\ey \Ry$. 
Thus, it only remains to prove that~$\vphi_{t}(x_{t-1})$ decreases in~$t$ up to certain error (since~$\vphi_{1}(x_{0}) \le \vphi(x_0)$). 
To this end, we proceed by induction. 
The base is obvious:~\cref{eq:recursive-bound} is satisfied when~$t=1$ since
\[
\vphi_1(x) - 2\Theta \le \bF_1(x,y) \le \vphi_1(x),
\]
and~$\vphi_{1}(x_{0}) \ge \vphi(x_0) - 2\bey \Ry$.
Now, assume that~\cref{eq:recursive-bound} was satisfied at steps~$\tau \in [t-1]$, so that our analysis of these steps was valid.
Then, by part~\proofstep{5} of the proof of~\cref{th:upper-bound}, in all these previous steps, including step~$t-1$, 
the saddle-point problem~\cref{eq:minmax-conceptual-updates} has been solved up to accuracy~$O(\ex^2/\Lxx)$ in primal gap:
\begin{eq}
\label{eq:primal-gap-bound}
\vphi_{\tau}(x_\tau) - \min_{x} \vphi_\tau(x) \le \frac{0.19\ex^2}{\Lxx}, \quad \tau \in [t-1],
\end{eq}
cf.~\cref{eq:iteration-acc-primal-gap}.
On the other hand, one can easily see that~$\vphi_{\tau}(x_{\tau-1}) \le \vphi_{\tau-1}(x_{\tau-1})$ for all~$\tau \in [\bTx]$, cf.~\cref{eq:minmax-conceptual-updates}.
Combining the two inequalities sequentially, we get
\[
\begin{ald}
\vphi_t(x_{t-1}) \le \vphi_{t-1}(x_{t-2}) + \frac{0.2\ex^2}{\Lxx} 
&\le \vphi(x_{0}) + \frac{0.2\bTx\ex^2}{\Lxx} 
\le \vphi(x_{0}) + 2\Gap + 4\ey\Ry.
\end{ald}
\]
Combining this with~\cref{eq:recursive-proof-first}, we arrive at~\cref{eq:recursive-bound}.  \hfill\proofbox

\odima{
\subsection{Proof of~\cref{prop:moreau-to-primal}}
First observe that~$\nabla \phi_{2L}(\wh x) = 2L (\wh x - x^+)$ for any~$\wh x \in X$ by Danskin's theorem. 
Thus,~$x^+ = \wh x - \tfrac{1}{2L}\nabla \phi_{2L}(\wh x)$. This implies the first claim of the Proposition: indeed,~$x^+ \in X$, so that~$x^+ = \Pi_{X}(x^+)$ and hence
$
\Dx(\wh x,\nabla\phi_{2L}(\wh x),2L) = \Ex(x,\nabla\phi_{2L}(\wh x),2L) = \|\nabla\phi_{2L}(\wh x)\|.
$
The first part of the second claim is obvious. For the second part, note that~\cref{def:moreau-envelope} is a convex minimization problem by the weak-convexity of~$\phi$, and the first-order optimality condition for it is
\begin{eq}
\label{eq:moreau-first-order-optimality}
\exists \xi \in \partial \phi(x^+): \;\; \lang \xi + 2L(x^+ - \wh x), x - x^+ \rang \ge 0, \;\; \forall x \in X.
\end{eq}
For such~$\xi$, and assuming that~$\|\nabla\phi_{2L}(\wh x)\| \le \ex$, we have that
\[
\begin{ald}
\Dx^2(x^+,\xi,2L) 
&= 4L \max_{x \in X} \left[ -\langle \xi, x - x^+ \rangle - L \|x - x^+\|^2 \right] \\
&\le 4L^2 \max_{x \in X} \left[ 2 \langle x^+ - \wh x, x - x^+ \rangle - \|x - x^+\|^2 \right] \le 4L^2 \|x^+ - \wh x\|^2 \le \ex^2.
\end{ald}
\]
Here we first used~\eqref{eq:moreau-first-order-optimality} and then the Cauchy-Schwarz inequality. 
\proofbox

\subsection{Proof of~\cref{lem:constrained-concavity}}
By concavity~$h(y') - h(y) \leq \langle  y' - y, \nabla h(y) \rangle$; thus, it suffices to prove that
\begin{eq}
\label{eq:constrained-first} 
\langle y' - y, \zeta \rangle \le 
L \langle y' - y, y^+ - y\rangle + \tfrac{1}{2L}[\Dy^2(y,-\zeta,L) - \Ey^2(y,-\zeta,L)]
\end{eq}
where~$y^+ := \proj_{Y}(y+\tfrac{1}{L}\zeta)$, for arbitrary~$y',y,\zeta \in Y,$ and~$L > 0$. 
Indeed, then~\cref{eq:constrained-concavity} follows by applying~\eqref{eq:constrained-first} to~$\zeta = \nabla h(y)$ so that~$\zeta^L(y)=L(y^+-y)$.
Now, observe that
\[
\label{eq:auxiliary}
\begin{aligned}
   \langle y' - y, \zeta \rangle 
= L\langle y' - y,  y^{+} - y  \rangle + \langle y' - y, \zeta - L (y^{+} - y)  \rangle
\end{aligned}
\]
and~$\langle y' - y, \zeta - L(y^{+} - y)  \rangle \le \langle y^{+} - y, \zeta - L(y^{+} - y)  \rangle$ by the projection lemma (see, e.g.,~\cite[Lem.~3.1]{bubeck2014theory}).
Finally,
\begin{eqnarray*}
&\quad 
	\langle y^{+} - y, \zeta - L(y^{+} - y)  \rangle 
=
	\langle \zeta, y^{+} - y \rangle - L \|y^{+} - y \|^2 \\ 
&\le 
	\max_{w \in Y} \left[ \langle \zeta, w - y \rangle - \tfrac{L}{2}\|w - y \|^2 \right] - \tfrac{L}{2} \|y^{+} - y\|^2
= 
    \tfrac{1}{2L}[\Dy^2(z,-\zeta, L) - \Ey^2(z,-\zeta, L)].
\; \proofbox
\end{eqnarray*}

\subsection{Proof of~\cref{prop:fne-to-moreau}}
By the second claim of~\cref{prop:moreau-to-primal}, we have
$
\nabla \vphi_{2\Lxx}(\wh x) = 2\Lxx(\wh x - x^+),
$
thus we can focus on bounding~$\tfrac{1}{2}\Lxx\|\wh x - x^+\|^2$.
To this end, the~$\Lxx$-strong convexity of the function~$\vphi(\cdot) + \Lxx\|\cdot - \wh x\|^2$ (minimized at~$x^+$) yields
$
\tfrac{1}{2}\Lxx\|\wh x - x^+\|^2 
\le \vphi(\wh x) - \vphi(x^+) - \Lxx \|x^+ - \wh x\|^2.
$
Moreover, we clearly have
\[
\vphi(\wh x) - \vphi(x^+) 
\le F(\wh x, y^o) - F(\wh x, \wh y) + F(\wh x, \wh y) - F(x^+, \wh y) 
\]
for~$y^o \in Y$ such as~$F(\wh x, y^o) = \vphi(\wh x)$.
Now, by the descent lemma (due to~\cref{ass:smoothness}) we get
\[
\begin{aligned}
	F(\wh x, \wh y) - F(x^+, \wh y) - \Lxx \|x^+ - \wh x\|^2
&\le 
	-\lang \gx F(\wh x, \wh y), x^+ - \wh x \rang - \tfrac{1}{2} \Lxx \|x^+ - \wh x\|^2 \\
&\le \tfrac{1}{2\Lxx}\Dx^2(\wh x, \gx F(\wh x, \wh y), \Lxx).
\end{aligned}
\]
On the other hand, applying~\cref{lem:constrained-concavity} to~$h(\cdot) = F(\wh x, \cdot)$ with~$L = \Lyy$ results in
\[
F(\wh x, y^o) - F(\wh x, \wh y) 
\le \Ry \hEy + \tfrac{1}{2\Lyy}{[\hDy{}^2 - \hEy{}^2]}.
\]
Combining the results obtained so far, we arrive at~\cref{eq:fne-to-moreau}. 
The second claim of the lemma follows by using that~$\hEy \le \hDy$, and requiring that~$\max[\hDy\Lxx\Ry, \hDy{}^2\Lxx/\Lyy] \le \ex^2.$
\proofbox
}

\section{Extension to non-Euclidean geometries}
\label{sec:non-euclidean}

Here we do \textit{not} assume the norm~$\|\cdot\|$ to be Euclidean (unless explicitly stated).

\subsection{Near-stationary points of a convex function}

The first challenge when extending~\cref{alg:fsp} to the non-Euclidean setup arises already in the sub-problem of finding a near-stationary point of a smooth and convex function.
Therefore, we first focus on this problem in isolation.
Given a norm~$\|\cdot\|$ on~$\R^d$ and its dual norm~$\|\cdot\|_*$, consider the problem of finding~$\veps$-first-order-stationary point~$\wh z \in \R^d$ of function~$f: \R^d \to \R$, i.e., such that~$\|\nabla f(\wh  z)\|_* \le \veps$. 
We assume that~$f$ is convex and has~$L$-Lipschitz gradient with respect to~$\|\cdot\|$, i.e.,
\[
\| \nabla f(z') - \nabla f(z) \|_* \le \|z' - z\|, \quad \forall z',z \in \R^d,
\]
and that at least one such~$\wh z$ belongs to the origin-centered $\|\cdot\|$-norm ball with~radius~$R$.

Recall that, in the Euclidean case, 
the recipe of Nesterov~\cite{nesterov2012howto} is to add the regularizer~$r_\veps(z) = \frac{\veps}{2R} \|z\|^2$, observing that the regularized function~$f_{\veps}$ has two properties:
\begin{itemize}
\item[$(i)$]
$f_{\veps}$ has~$(L+\veps)$-Lipschitz gradient (since~$r_\veps(z)$ has~$\veps$-Lipschitz gradient) and is~$\veps$-strongly-convex (since $r_\veps(z)$ is strongly convex). 
\item[$(ii)$] 
The gradient~$\nabla f_\veps$ uniformly approximates~$\nabla f$ with respect to~$\|\cdot\|_* = \|\cdot\|_2$:
\begin{eq}
\label{eq:lipschitz-gradient-general}
\| \nabla f_\veps(z) - \nabla f(z) \|_2 \le \frac{\veps \|z\|_2}{R} \le \veps.
\end{eq}
\end{itemize}
Property~$(ii)$ allows to search for approximate stationary points of~$f_\veps$ instead of~$f$, whereas~$(i)$ guarantees that restarted FGM (\cref{alg:fgm-restart}) finds such a point in~$\wt O(\sqrt{\kappa})$ queries of~$\nabla f(\cdot)$ in total, where~$\kappa = L/\veps$, which results in the complexity bound 
\begin{eq}
\label{eq:non-euclidean-complexity-euclidean}
T_\veps = \wt O\big(\sqrt{{LR}/{\veps}}\big).
\end{eq}
This complexity bound is optimal up to a logarithmic factor in the Euclidean case~\cite{nesterov2012howto}. 

In the setup with a non-Euclidean proximal geometry, one would expect the complexity bound~\eqref{eq:non-euclidean-complexity-euclidean} to be preserved. 
More precisely, assume that the norm~$\|\cdot\|$, now not necessarily Euclidean, admits a \textit{distance-generating function (\dgf)}~$\omega: \ZZ \to \R$ replacing the squared norm~$\tfrac{1}{2} \|\cdot\|_2^2$ in the Euclidean case, with the following three properties~(see,~e.g.,~\cite{optbook1,nesterov2013first,algorec2018arxiv} and references therein):\footnotemark
\footnotetext{Here we first focus on the unconstrained setup for the sake of simplicity; the case where~$z$ lives on a ``simple'' convex body in~$\R^d$ can be treated in a similar vein, and is postponed to~\cref{sec:bregman-constrained}.}

1) The function~$\omega(\cdot)$ is convex, admits a continuous selection of subgradients (denoted~$\nabla \omega(z)$ later on), and has strong convexity modulus~$1$ w.r.t.~$\|\cdot\|$.

2) One can easily solve (explicitly or to high accuracy) optimization problems of the form~$\min_{z'} [\lang \zeta, z' \rang + \omega(z')]$, where~$\zeta$ is an arbitrary linear form (i.e., element of the dual space identified with~$\R^d$ by Riescz theorem), and~$\lang \cdot, \cdot \rang$ is the duality pairing (identified with the canonical dot product on~$\R^d$). Equivalently, one requires computational tractability of the problem
\[
\min_{z' \in \R^d} [\lang \zeta, z' \rang + D_{\omega}(z',z)],
\]
where~$D_{\omega}(z',z) := \omega(z') - \omega(z) - \lang \nabla \omega(z), z' - z \rang$ is the Bregman divergence generated by~$\omega$.
The fulfillment of these requirements is guaranteed by working with~\dgf's that are coordinate-separable (such as entropy on the non-negative orthant or~$\|\cdot\|_p^p$ for~$p \ge 1$) or ``quasi-separable'' (e.g., compositions of a separable function and a monotone map on~$\R$), such as~$\|\cdot\|_p^2$ with~$p \ge 1$.
 
3) Finally, we assume that~$\omega$ is minimized at the origin (and~$\omega'(x) = 0$ is included in the continuous selection ob sugradients), and satisfies the following \textit{quadratic growth} condition: the {\em $\omega.$-radius} functional~$\Omega[\cdot]$, defined as
\[
\Omega[Z] := \max_{z \in Z} \omega(z) - \min_{z' \in Z} \omega(z'),
\]
for compact subsets of~$\R^d$, satisfies
\begin{eq}
\label{eq:quadratic-growth}
\Omega[Z_r(0)] \le r^2 \wt O_d(1), \quad \forall r \ge 0,
\end{eq}
where~$Z_r(z) := \{ z': \|z' - z\| \le r\}$, and~$\wt O_d(1)$ is a logarithmic factor in~$d$.
In other words,~$\Omega[Z_r]$ grows as the squared radius of the~$\|\cdot\|$-ball, mimicking the squared norm~$\tfrac{1}{2}\|\cdot\|^2$ in this respect. 
Note also that the same bound holds for~$D_{\omega}(z, 0) \le \Omega[Z_r(0)]$ for all~$z \in Z_r(0)$. Moreover, these conditions can be ``re-centered'' to arbitrary point~$z_0$ by replacing~$\omega(\cdot)$ with the shifted~\dgf
\begin{eq}
\label{eq:dgf-recentering}
\omega_{z_0}(\cdot) := \omega(z - z_0)
\end{eq}
which is minimized at~$z_0$ and satisfies~\cref{eq:quadratic-growth} with~$Z_r(z_0)$ instead of~$Z_r(0)$; the previous properties hold for~$\omega_{z_0}$ as well. 
Here we note that the ``slow growth'' property is not required to obtain convergence guarantees in terms of the~$\omega$-radius; rather, it is needed to ``translate'' such guarantees to those in terms of the~$\|\cdot\|$-norm distance to optimum. 
Another remark is that the balls~$Z_r$ here are only allowed to be centered in the origin (i.e., in the minimum of~$\omega$), which makes the condition significantly less restrictive than that in~\cite{dang2015stochastic} where~\eqref{eq:quadratic-growth} is required to hold for balls with arbitrary centers, not only those centered at the~\dgf~minimizer. 
Note that the latter condition implies the Lipschitzness of~$\nabla \omega$ with respect to~$\|\cdot\|$, whereas the former does not; we will revisit this circumstance in~\cref{sec:bregman-prox-point}.

We call any~\dgf~satisfying the above three properties {\em compatible with~$\|\cdot\|$}. 
Whenever one can find a compatible~\dgf, the usual recipe is to modify the ``Euclidean'' algorithm by replacing the Euclidean prox-mapping~\cref{eq:grad-map} with its generalization:
\begin{eq}
\label{eq:grad-map-general}
\prox_{z,\omega}(\zeta) := \argmin_{z'} \left[ \lang \zeta, z'\rang + D_{\omega}(z',z) \right],
\end{eq}
which corresponds to replacing the gradient descent step with so-called mirror descent step (\cite{nemirovsky1983problem}) -- ``steepest descent'' with respect to the~\dgf~that takes into account the geometry of~$\|\cdot\|$. 
For many standard primitives in convex optimization, such a recipe results in the desirable outcome: the distance to optimum~$R$ and the Lipschitz constant~$L$ get replaced with their~$\|\cdot\|$-norm counterparts. In particular, this is the case for FGM (\cref{alg:fgm}) as Theorem~\ref{th:devolder} generalizes almost verbatim.

\begin{theorem}[{\cite[Thm.~5 and~Eq.~(42)]{devolder2014first}}]
\label{th:devolder-general}
Assume~$f$ is convex, has~$L$-Lipschitz gradient with respect to the norm~$\|\cdot\|$, cf.~\cref{eq:lipschitz-gradient-general}, is and minimized at~$z^*$ such that~$\|z^* - z_0 \| \le R$.
Consider running Algorithm~\ref{alg:fgm}, with prox-mappings in lines 4 and 8 replaced by the generalized prox-mapping~\cref{eq:grad-map-general} with respect to the~\dgf~$\omega_{z_0}$ (the re-centered to~$z_0$ compatible~\dgf,~$\omega$,~cf.~\cref{eq:dgf-recentering}), with stepsize~$\gam = 1/L$, and~$\delta$-inexact oracle in the sense of~\cref{def:inexact-oracle} (with~$\|\cdot\|$ being the given norm). 
Then 
\begin{eq}
\label{eq:fgm-convergence-general-with-error}
f(z_T) - f(z^*) \le \frac{4 L \Omega}{T^2} + 2\delta T \le \frac{4\wt O_d(1) LR^2}{T^2} + 2\delta T,
\end{eq}
where~$\Omega := \Omega[Z_R(z_0)]$ is the~$\omega$-radius of the~$z_0$-centered ball containing~$z^*$.
Thus,
\begin{eq}
\label{eq:fgm-convergence-general}
f(z_T) - f(z^*) \le \frac{5L \Omega}{T^2} \le \frac{5 \wt O_d(1) LR^2}{T^2} 
\;\; \text{whenever} \;\;
\delta \le \delta_T := \frac{L\Omega}{2T^3}.
\end{eq}
\end{theorem}
Returning to our problem of finding a near-stationary point of~$f(\cdot)$, the reasonable approach would be to regularize~$f$ with the term
\begin{eq}
\label{eq:regularizer-general}
r_{\veps}(z) = \frac{\veps \omega(z)}{\sqrt{\Omega}}, 
\end{eq}
which reduces to~$\frac{\veps}{2R} \|z\|_2^2$ in the Euclidean setup with~$\tfrac{1}{2}\|\cdot\|_2^2$ used as \dgf \,
However, we immediately see that neither of the properties ~$(i), (ii)$ remain valid.
\begin{itemize}
\item 
Indeed, while the regularized function~$f_\veps(z)$ is strongly convex with respect to~$\|\cdot\|$, its gradient can be non-Lipschitz: in fact, the existence of functions that are \textit{strongly convex and smooth at the same time,} with near-constant condition number, is quite special for the Euclidean norm. 
\item
As for the property~$(ii)$, it is again a ``fortunate coincidence'' that in the Euclidean case~$\nabla \omega(z) \equiv z$ and~$\|\cdot\|_* \equiv \|\cdot\|$, whence~$\|\nabla r_\veps (z) \|_* \le \veps$ on~$Z_R(0)$.
\end{itemize}

The first of these issues is easy to fix: instead of treating~$f_\veps$ as a smooth function, which it is not anymore, 
one can treat it as a composite function with~$L$-smooth part~$f$ and a non-smooth but ``simple'' term~$r_\veps$, simplicity being guaranteed by the compatibility of~$\omega$. 
As such, one can exploit the ``tolerance'' of Algorithm~\ref{alg:fgm} to such composite objectives: one can use the inexact gradient oracle for~$f$, rather than for~$f_\veps$, instead incorporating~$r_{\veps}$ into the prox-mapping, i.e., replacing~\cref{eq:grad-map-general} with
\[
\begin{ald}
\prox_{z,\omega,\veps}(\zeta) 
&:=\argmin_{z'} \left[ \lang \zeta, z'\rang + D_{\omega}(z',z) + \frac{\veps}{\sqrt{\Omega}} D_{\omega}(z',z_0)\right].
\end{ald}
\]
As shown in~\cite[Sec.~6.3 and Thm.~8]{devolder2011stochastic},~\cref{th:devolder-general} generalizes to this most general setup: the guarantees~\cref{eq:fgm-convergence-general-with-error}-\cref{eq:fgm-convergence-general} remain valid, with~$f$ in the left-hand side replaced by~$f_\veps$, and~$z^*$ being the minimizer of~$f_\veps$. 
As a result, using the strong convexity of~$f_\veps$, we can proceed with the same restart scheme as before (\cref{alg:fgm}). 
We now state the appropriate modification of~\cref{cor:fgm-restart-bound}.

\begin{corollary}
\label{cor:fgm-restart-bound-general}
Let~$f_{\lam}$ be a composite function given by
\[
f_{\lam\sqrt{\Omega}}(z) := f(z) + \lam D_{\omega}(z,z_0),
\]
with~$\lam \ge 0$, and~$f$ having~$L$-Lipschitz gradient with respect to~$\|\cdot\|$. 
Given~$\veps > 0$, run~\cref{alg:fgm-restart} on ~$f_{\lam \sqrt{\Omega}}$ with~$\gam = 1/L$, parameters~$T,S$ satisfying
\begin{eq}
\label{eq:fgm-restart-params-general}
T \ge \sqrt{40 \wt O_d(1) L/\lambda}, \quad S \ge \log_{2}\left({3L\sqrt{\Omega}}/{\veps} \right), 
\end{eq}
where~$\wt O_d(1)$ is the logarithmic factor in~\cref{eq:fgm-convergence-general-with-error}, and~$\delta \le \delta_T$, cf.~\eqref{eq:fgm-oracle-accuracy-bound}.
Then~$z^{S}$ satisfies
\begin{eq}
\label{eq:fgm-approximation-bounds-general}
\begin{ald}
\|z^{S}- z^*\| \le \frac{\veps}{3L}, \;\; 
f_{\lam \sqrt{\Omega}}(z^{S}) - f_{\lam \sqrt{\Omega}}(z^*) \le \frac{\veps^2}{18L}, \;\; 
\|\nabla f(z^S) - \nabla f(z^*) \|_* \le \frac{\veps}{3}.
\end{ald}
\end{eq}
\end{corollary}
\begin{proof}
Note that, with given~$T$, we ensure that~$R_{s} := \|z^{s} - z^*\|$ satisfies
\[
R_s \stackrel{\eqref{eq:fgm-convergence-general}}{\le}  \sqrt{\frac{2}{\lam} \cdot \frac{5\wt L\Omega[Z_{R_{s-1}}(z^{s-1})]}{T^2}} \le \sqrt{\frac{10\wt O_d(1) LR_{s-1}^2}{\lam T^2}} \le \frac{R_{s-1}}{2}.
\]
Here the first transition relied on the fact that~$\omega$ is re-centered to~$z_{s-1}$ at~$s$-th epoch, and the second transition used the quadratic growth condition~\cref{eq:quadratic-growth}. 
This gives the first inequality in~\cref{eq:fgm-approximation-bounds-general}
The second inequality  can be verified as in the proof of~\cref{cor:fgm-restart-bound} (with~$\Omega$ replacing~$R^2$), and the last one follows by smoothness.
\end{proof}
We see that the first of the two issues with regularization is solved: we simply run~\cref{alg:fgm-restart} on~$f_{\veps}$.
Alas, the second issue is still present: while~$\nabla f(z^S)$ approximates~$\nabla f(z^*)$, where~$z^*$ minimizes~$f_\veps$, we cannot guarantee that~$\|\nabla f(z^*)\|_*$ is small: indeed,~$\|\nabla f(z^*)\|_* \le \veps$ is equivalent to
\begin{eq}
\label{eq:potential-gradient-bound}
\sup_{z \in Z_R} \|\nabla\omega(z)\|_*^2 \le \Omega,
\end{eq}
but this latter condition cannot be guaranteed from the compatibility properties of~$\omega$. 
In fact, in the constrained setup, where minimization has to be performed on a convex body~$Z \subset \R^d$,~\cref{eq:potential-gradient-bound} breaks for the important class of Legendre~\dgf's -- those with gradients diverging on the boundary of the feasible set~\cite{bauschke2016descent}.\footnotemark
\footnotetext{E.g., in the ``simplex'' setup, where the norm is~$\|\cdot\|_1$, and~\dgf~is the negative entropy~$h(z)  = \sum_{i \in [d]} z_i \log(z_i)$ on the probability simplex~$\Delta_d \subset \R^d$.
The appropriate modification of~\cref{eq:potential-gradient-bound},~$\sup_{z \in \Delta_d} \|\nabla h(z)\|_\infty^2 \le \log(d)$, cannot be valid  since the left-hand side is infinite.}
However, in the absence of constraints, or for non-Legendre potentials in the constrained case,~\cref{eq:potential-gradient-bound} can sometimes be guaranteed. Next we consider one such example relevant in practice.

\paragraph{Regularization with~$\|\cdot\|_p^2$}
Let the norm of interest be~$\|\cdot\|_1$ with the dual norm~$\|\cdot\|_{\infty}$. It is well-known (see, e.g.~\cite{nesterov2013first}) that, for any~$d \ge 3$, the function
\begin{eq}
\label{eq:lp2-dgf}
\omega(z) = \frac{C_{d}}{2} \|z\|_p^2 
\quad \text{with} \;\; p = 1 + \frac{1}{\log(d)} \;\; \text{and} \;\; C_d = \exp\left( \frac{\log d - 1}{\log d + 1}\right)
\end{eq}
is a compatible~\dgf~for~$\|\cdot\|_1$; in particular,~$\omega(z)$ is~$1$-strongly convex on~$\R^d$ with respect to~$\|\cdot\|_1$, and~$\Omega[Z_1] \le c\log(d)$ for some universal constant~$c$ (with a matching lower bound).
At the same time,~\cref{eq:potential-gradient-bound} can be easily verified: 
$\nabla \omega(z) = C_d \|z\|_p^{2-p} z^{p-1},$
where the coordinates of~$z^{p-1} \in \R^d$ are the~$(p-1)$-th powers of the coordinates of~$z$ (with the signs preserved).
As a result,
\small
\begin{eq}
\label{eq:lp2-grad-bound}
\begin{ald}
\sup_{\|z\|_1 \le 1} \|\nabla \omega(z)\|_{\infty} 
&= C_d \sup_{\|z\|_1 \le 1} \|z\|_p^{2-p} \| z \|_{\infty}^{p-1} 
\le C_d \sup_{\|z\|_1 \le 1} \|z\|_p
\le \sqrt{2cC_d\log(d)} 
\end{ald}
\end{eq}
\normalsize
where we first used that~$\|z\|_{\infty} \le \|z\|_p$ and then used the bound~$\Omega[Z_1] \le c \log(d)$. 
Thus,~\cref{eq:potential-gradient-bound} is verified, 
so~$\|\cdot\|_p^2$-regularization only perturbs the gradient up to~$O(\veps)$.

\subsection{Constrained case}
\label{sec:bregman-constrained}
We have just seen that in the unconstrained scenario, one can indeed efficiently approximate first-order stationary points of a convex function -- at least in the~$\ell_1$-geometry setup. Let us now demonstrate that this result can be extended to the constrained scenario.
Namely, we now incorporate into the problem a set~$Z \in \R^d$, assumed to be convex, compact, and ``prox-friendly'':
one must be able to efficiently compute the prox-mapping with respect to~$Z$, defined as
\begin{eq}
\label{eq:grad-map-general-constrained}
\prox_{z,Z,\omega}(\zeta) := \argmin_{z' \in Z} \left[ \lang \zeta, z'\rang + D_{\omega}(z',z) \right];
\end{eq}
note that this is satisfied when~$Z$ is a ``simple'' set such as~$\ell_p$-ball or a simplex.
Accordingly, we modify the \dgf~compatibility requirements, now only requiring strong convexity on~$Z$. 
It is known from~\cite{devolder2014first} that~\cref{th:devolder-general} extends almost word-for-word to this setting, with the prox-mapping~\cref{eq:grad-map-general} replaced with~\cref{eq:grad-map-general-constrained}, and~$\Omega[Z_R]$ replaced with~$\Omega = \Omega[Z]$.
Furthermore, the first two inequalities in~\cref{eq:fgm-approximation-bounds-general} are preserved, under the same premise~\cref{eq:fgm-restart-params-general}. 
Now, let us define the natural ($\omega$-adapted) stationarity measure~$\Dzw$ by
\small
\begin{align}
\label{def:our-measure-general}
\Dzw^2(z, \zeta, L) := 2L \max_{z' \in Z} \left[-\lang \zeta, z' - z \rang  - L D_{\omega}(z',z) \right],
\end{align}
\normalsize
We can easily verify that, under~\cref{eq:fgm-restart-params-general}, one has 
\small
\begin{eq}
\label{eq:fgm-near-stationarity-general-constrained}
\Dzw(z^{S}, \nabla f_{\lam \sqrt{\Omega}}(z^S), L+\lam) \le \frac{\veps}{3}\sqrt{\frac{L+\lam}{L}}.
\end{eq}
\normalsize
Indeed,  the argument mimics that in~\cref{eq:-fgm-near-stationarity-argument}: 
\small
\[
\begin{ald}
\Dzw^2(z^{S},\nabla f_{\lam \sqrt{\Omega}}(z^S), L+\lam) 
=-2(L+\lam) \min_{z \in Z} \big[\langle \nabla f_{\lam \sqrt{\Omega}}(z^S), z - z^S \rangle + (L+\lam) D_{\omega}(z,z^S) \big];
\end{ald}
\]
\normalsize
on the other hand, for any~$z \in Z$ one has
\small
\[
\small
\begin{aligned}
&f_{\lam \sqrt{\Omega}}(z) - f_{\lam \sqrt{\Omega}}(z^S)= f(z) - f(z^S) + \lam[D_{\omega}(z,z_0) - D_{\omega}(z^S,z_0)] \\
&\le \langle \nabla f(z^S), z-z^S \rangle + \frac{L}{2} \| z - z^S\|^2 + \lam[D_{\omega}(z,z_0) - D_{\omega}(z^S,z_0)] \\
&\le \langle \nabla f(z^S), z-z^S \rangle + L D_\omega(z,z^S) + \lam[D_{\omega}(z,z_0) - D_{\omega}(z^S,z_0)] \\
&= \langle \nabla f(z^S), z-z^S \rangle + L D_\omega(z,z^S) + \lam[D_{\omega}(z,z^S) + \langle \nabla \omega(z^S) - \nabla \omega(z_0), z - z^S\rangle] \\ 
&= \langle \nabla f_{\lam \sqrt{\Omega}}(z^S), z-z^S \rangle + (L+\lam) D_\omega(z,z^S),
\end{aligned}
\normalsize
\]
\normalsize
where we first used the smoothness of~$f$, then the~$1$-strong convexity of~$\omega$, and finally, the well-known three-point identity for the Bregman divergence~(see,~e.g.,~\cite[Eq.~(4.1)]{bubeck2014theory}).
Minimizing both sides over~$z \in Z$ and recalling that~$f_{\lam \sqrt{\Omega}}(z^S) - f_{\lam \sqrt{\Omega}}(z^*) \le \veps^2/(18L)$, we arrive at~\cref{eq:fgm-near-stationarity-general-constrained}. \qed

Applying~\eqref{cor:fgm-restart-bound-general} with~$\lambda = \veps/\sqrt{\Omega}$, i.e., to the regularized function~$f_\veps$, cf.~\cref{eq:regularizer-general}, we see that that one can obtain~$O(\veps)$-stationary point -- either in the sense of the dual gradient norm in the unconstrained case, or in the sense of~$\Dzw^2(\cdot,\cdot,\cdot)$ criterion -- in~$\wt O(\sqrt{LR/\veps})$ prox-mapping computations, by running appropriately generalized version of~\cref{alg:fgm-restart}. 

\begin{rem}
Using the optimality conditions in~\cref{def:our-measure-general}, one can verify that 
\begin{eq}
\label{eq:our-measure-bregman-characterization}
\begin{ald}
\Dzw^2(z, \nabla f(z), L) 
&\ge 2L^2 D_{\omega}(z, \nabla\omega^*_Z[\nabla\omega(z) - \tfrac{1}{L} \nabla f(z)]) \\
&[= 2L^2 D_{\omega^*_Z} (\nabla \omega(z) - \tfrac{1}{L} \nabla f(z), \nabla \omega(z))]
\end{ald}
\end{eq}
with equality in the unconstrained case. Here,~$\omega^*_Z$ is the Fenchel dual of~$\omega$ on~$Z$, i.e.,
\[
\omega_Z^*(\zeta) := \max_{z \in Z} \left[ \lang \zeta, z \rang - \omega(z)\right], \quad \forall \zeta \in \R^d,
\] 
and~$\nabla\omega^*_Z[\nabla\omega(z) - \tfrac{1}{L} \nabla f(z)]$ is the mirror descent update from~$z$. 
The second representation in~\cref{eq:our-measure-bregman-characterization} is by the standard properties of the Bregman divergences (\cite{rockafellar2015convex}).
From it, noting that~$\nabla \omega^*_Z$ is~$1$-Lipschitz with respect to~$\|\cdot\|_*$, we conclude that~$\Dzw(z, \nabla f(z), L)$ {\em under-estimates} the dual gradient norm~$\|\nabla f(z) \|_*$ in the unconstrained setup, this estimate only being tight in the Euclidean case, i.e., when~$\omega(z) = \tfrac{1}{2}\|z\|_2^2$. 
On the other hand, from the first representation we see that~$\Dzw(z, \nabla f(z), L)$ {\em over-estimates} the proximal gradient norm measure~$\Ezw(z, \nabla f(z), L)$ defined by
\[
\Ezw^2(z, \nabla f(z), L) = L^2 \left\|z - \nabla \omega^*_Z[\nabla\omega(z) - \tfrac{1}{L} \nabla f(z)] \right\|^2.
\]
Thus,~$\Dzw$ corresponds to a stronger criterion than~$\Ezw$ in the constrained case; in the unconstrained case the two measures coincide, and the resulting criterion is {\em weaker} than the gradient norm one (unless the norm is Euclidean).
\end{rem}

Summarizing the results of this section, we see that one of the two key ``computational primitives'' in our framework -- the search of a near-stationary point of a smooth and concave function -- extends to the~$\ell_1$-geometry with distance-generating function given by~\cref{eq:lp2-dgf}, where the accuracy can be measured by the~$\Dzw$ measure~(cf.~\cref{def:our-measure-general}) or by the dual gradient norm in the unconstrained case. 
Thus, we have extended the results of~\cref{sec:fgm-inexact}. 
Our next goal is to similarly extend the results of~\cref{sec:ppm-via-fgm}, i.e., to implement the non-Euclidean proximal point algorithm with inexact iterations. 

\subsection{Bregman proximal point algorithm}
\label{sec:bregman-prox-point}
Given a function~$\phi: X \to \R$ with~$L$-Lipschitz gradient with respect to the norm~$\|\cdot\|$, where~$X \subseteq \R^d$ is convex and ``prox-friendly'' with respect to a compatible with~$\|\cdot\|$ \dgf~$\omega$, the goal is to find a point~$\wh x \in X$ such that~$\Dxw(\wh x, \nabla \phi(\wh x), L) \le \veps$.
As in~\cref{sec:ppm-via-fgm}, we will achieve this result via proximal point updates implemented using~\cref{alg:fgm-restart}. 
First, we define the Bregman proximal point operator following~\cite{nemirovski2004prox}:
\begin{eq}
\label{def:prox-map-bregman}
x \mapsto x^+_{\gam\phi, X, \omega}(x) := \argmin_{x' \in X} \left[ \phi (x') + \tfrac{1}{\gam} D_{\omega}(x',x) \right];
\end{eq}
note that the objective in~\cref{def:prox-map-bregman} is~$1/\gamma$-strongly convex with respect to~$\|\cdot\|$. 
We denote~$x^+ := x^+_{\gam\phi, X, \omega}(x)$ for brevity, and fix~$\gam = \frac{1}{2L}$.
The optimality condition reads
\begin{eq}
\label{eq:ppm-foo-bregman}
\lang \tfrac{1}{2L}\nabla\phi(x^+) + \nabla \omega(x^+) - \nabla \omega(x), x'-x^+ \rang \ge 0, \quad \forall x' \in X.
\end{eq}
Following~\cref{sec:ppm-via-fgm}, we first analyze the exact updates
$
x_{t} = x^+_{\phi/(2L), X, \omega}(x_{t-1}).
$
By~\cref{def:prox-map-bregman} we have 
$
\phi(x_{t-1}) \ge \phi(x_t) + 2L D_\omega(x_{t},x_{t-1}) 
$
which allows to mimic~\cref{eq:prox-point-argument}:
\begin{eq}
\label{eq:prox-point-argument-bregman-min}
\min_{t \in [T]} \|x_t - x_{t-1} \|^2 \le 2D_{\omega}(x_t, x_{t-1}) \le \frac{2}{T} \sum_{t \in [T]} D_{\omega}(x_t, x_{t-1}) \le \frac{\Gap}{LT}.
\end{eq}
On the other hand, we can bound the stationarity measure proceeding as in~\cref{eq:stationarity-by-closeness}:
\small
\begin{eq}
\label{eq:prox-point-argument-bregman-chain}
\hspace{-0.2cm}
\begin{ald}
\Dxw^2(x^+,\nabla\phi(x^+), L) 
& \equiv 2L \max_{x' \in X} \left[-\lang \nabla \phi(x^+), x' - x^+ \rang - L D_{\omega}(x', x^+)\right] \\
&\le 2L^2 \max_{x' \in X} \left[ 2\lang \nabla \omega(x^+) - \nabla\omega(x), x' - x^+ \rang - D_{\omega}(x', x^+) \right] \\
&\le 2L^2 \max_{x' \in X} \left[ 2\| \nabla \omega(x^+) - \nabla\omega(x)\|_*^2 + \tfrac{1}{2}\| x' - x^+\|^2 - D_{\omega}(x', x^+) \right] \\
&\le 4L^2 \| \nabla \omega(x^+) - \nabla\omega(x)\|_*^2,
\end{ald}
\end{eq}
\normalsize
where we first used Young's inequality and then the strong convexity of~$\omega$. 
Note that in the unconstrained case, and with~$\Dxw^2(x^+,\nabla\phi(x^+), L)$ replaced by~$\|\nabla f(x^+)\|_*^2$, the bound~\cref{eq:prox-point-argument-bregman-chain} becomes an equality; on the other hand, we have not been able to find a tighter bound for~$\Exw^2(x^+,\nabla\phi(x^+), L)$; all this indicates that~\cref{eq:prox-point-argument-bregman-chain} is likely unimprovable in general.
Now, the inequalities~\cref{eq:prox-point-argument-bregman-min} and~\cref{eq:prox-point-argument-bregman-chain}, when combined together, imply that, in order to proceed as in the Euclidean case, one must require that the~\dgf~$\omega$ is \textit{smooth} on~$X$ with respect to~$\| \cdot \|$, i.e., for some~$\ell_{X,\omega} \ge 1$ one has
\begin{eq}
\label{eq:dgf-smoothness}
\| \nabla \omega(x'') - \nabla\omega(x')\|_* \le \ell_{X,\omega} \|x'' - x' \|, \quad \forall x',x'' \in X.
\end{eq}
Indeed, when combined with~\cref{eq:prox-point-argument-bregman-min}--\cref{eq:prox-point-argument-bregman-chain}, this implies, after~$T$ exact updates, that 
\[
\min_{t \in [T]} \Dxw(x_t,\nabla\phi(x_t), L) \le 2\ell_{X,\omega} \sqrt{\frac{L\Gap}{T}},
\]
i.e., the same convergence rate as in the Euclidean case (up to the extra factor~$\ell_{X,\omega}$). 
Moreover, as in the Euclidean case, this argument preserves ``robustness'' to errors in~\cref{def:prox-map-bregman}.
Indeed, denoting~$\phi_{L,x,\omega}(\cdot)$ the objective in~\cref{def:prox-map-bregman}, assume that~$\wt x^{+}$ satisfies 
\[
\phi_{L,x,\omega}(\wt x^+) \le \phi_{L,x,\omega}(x^+) + \frac{\veps^2}{24L} 
\quad \text{and} \quad
\Dxw(\wt x^+, \nabla \phi_{L,x,\omega}(\wt x^+), L+\lam) \le \frac{\veps}{2},
\]
cf.~\cref{eq:ppm-objective-approx}--\cref{eq:ppm-stationarity-approx}. 
As we know from the results of~\cref{sec:bregman-constrained}, this can be guaranteed by running~\cref{alg:fgm-restart} on~$\phi_{L,x,\omega}$ with appropriately chosen parameter values. On the other hand, the sequence~$\wt x_t = \wt x^+_{\phi/2L,X,\omega}(\wt x_{t-1})$ satisfies the counterpart of~\cref{eq:prox-point-argument-bregman-min}:
\[
\min_{t \in [T]} \|\wt x_t - \wt x_{t-1} \|^2 \le 2D_{\omega}(\wt x_t, \wt x_{t-1}) \le \frac{2}{T} \sum_{t \in [T]} D_{\omega}(\wt x_t, \wt x_{t-1}) \le \frac{\Gap}{LT} + \frac{\veps^2}{24L^2},
\]
and that of~\cref{eq:prox-point-argument-bregman-chain}:
\[
\begin{ald}
&\Dxw^2(\wt x^+, \nabla \phi(\wt x^+), 2L+\lam) \\
\le  & \; 2(2L+\lam) \max_{x' \in X} \left[- \lang \nabla \phi_{L,x,\omega}(\wt x^+), x' - \wt x^+ \rang - (L+\lam) D _{\omega}(x',x) \right] \\
&\quad\quad\quad+ 2L(2L+\lam) \max_{x' \in X} \left[ 2\lang \nabla\omega(\wt x^+) - \nabla\omega(x), x' - \wt x^+ \rang - D_{\omega}(x',x) \right]\\
\le & \; 2 \Dxw^2(\wt x^+, \nabla \phi_{L,x}(\wt x^+), L+\lam) + 8L(L+\lam) \|\nabla\omega(\wt x^+) - \nabla\omega(x)\|_*^2 \\
\le &\; {\veps^2}/{2} + 8L(L+\lam) \|\nabla\omega(\wt x^+) - \nabla\omega(x)\|_*^2,
\end{ald}
\]
where we applied Young's inequality and strong convexity.
This allows to mimic~\cref{eq:stationarity-guarantee-approx}:
\[
\begin{ald}
\min_{t \in [T]} \Dx(\wt x_t,\nabla \phi(\wt x_t), L)  
&\le \ell_{X,\omega} \sqrt{\frac{8(L+\lam)\Gap}{T} + \frac{5\veps^2}{6}\left(\frac{L+\lam}{L}\right)^2} 
\end{ald}
\]
Taking~$\lam = L$, we arrive at the desired complexity estimate~$O(L\Gap/\veps^2)$. 

\subsection{On restrictiveness of \dgf~smoothness} 
While trivially satisfied in the Euclidean case with~$\ell_{X,\omega} = 1$ for any~$X$,  the smoothness assumption~\cref{eq:dgf-smoothness}  -- which is also made, e.g., in~\cite{dang2015stochastic,zhao2020primal} -- is strong in the general Bregman scenario. 
For example, the negative entropy~$h(x) = \sum_{i \in [d]} x_i \log(x_i)$, perhaps the most common choice of a non-Euclidean~\dgf, is {\em not} smooth on its domain, the probability simplex~$\Delta_d$. \footnotemark
\footnotetext{In fact, it is easy to see that no Legendre~\dgf, i.e., such that~$\|\nabla \omega(x)\|_* \to \infty$ when~$x \to \partial\,\textup{dom}(\omega)$, can satisfy~\cref{eq:dgf-smoothness} with~$X = \textup{dom}(\omega)$ for any finite~$\ell_{X,\omega}$.}
Likewise, the previously considered~$\|\cdot\|_p^2$-function (cf.~\cref{eq:lp2-dgf}) does not satisfy~\cref{eq:dgf-smoothness} with respect to~$\|\cdot\|_1$ on the set~$X = \{x \in \R^d: \|x\|_1 \le r\}$ for any finite~$\ell_{X,\omega}$, however small is~$r > 0$, unless when~$d = 1$.\footnotemark\;
More generally, as implied by~\cite[Theorem 3.5 and its omitted dual version]{borwein2009uniformly} in combination with~\cite[p.~469]{ball1994sharp}, there is no function simultaneously strongly convex and smooth, with dimension-independent condition number, with respect~$\ell_p$-norm for~$p \ne 2$. 
Still, it would be interesting to exhibit~$(X,\|\cdot\|, \omega)$ with~$\|\cdot\|$-compatible~$\omega$ for which~\cref{eq:dgf-smoothness} holds with {\em moderate}~$\ell_{X,\omega}$. 
\footnotetext{Indeed, by explicitly computing the Hessian~$\nabla^2 \omega(x)$ defined almost everywhere on~$\R^d$, we observe that its first diagonal element ``explodes'':~$[\nabla^2 \omega(x)]_{11} \to \infty$ when~$x = u e_1 + (1-u) e_2$ with~$u \to 0$.}

Alternatively, we may consider circumventing~\cref{eq:dgf-smoothness} by discarding Bregman divergences and working directly with the norm. 
Indeed, using conjugacy of~$\half\|\cdot\|^2$ and~$\half\|\cdot\|_*^2$ one can derive the~$O(L\Gap/\veps^2)$ convergence rate for minimizing the gradient norm up to~$\veps$ by {\em steepest descent} with respect to the norm~$\|\cdot\|$, i.e., replacing the Bregman divergence in~\eqref{eq:grad-map-general} by~$\half\|\cdot\|^2$. The resulting prox-mapping is tractable whenever~$\|\cdot\|^2$ is a ``simple'' function, which is the case, e.g., for~$\|\cdot\| = \|\cdot\|_p$ with~$p \ge 1$. 
 Likewise, the proximal point operator, when adjusted in this manner, remains tractable as~$\|\cdot\|_p^2$ is~$O(1)$-strongly convex with respect to~$\|\cdot\|_p$ when~$1 < p \le 2$. Thus, the results of~\cref{sec:bregman-prox-point} extend to such~$\|\cdot\|^2$-regularized proximal point algorithm.

\end{document}